\documentclass{article}
\usepackage{amsmath,amssymb,amsthm}
\usepackage{geometry}
\usepackage{hyperref}
\usepackage{authblk}
\usepackage{xcolor}
\usepackage{graphicx}
\usepackage{cleveref}

\usepackage[numbers,sort&compress]{natbib}
\title{Divergence-free Linearized Neural Networks: \\Integral Representation and Optimal Approximation Rates\thanks{Code is available at: \url{https://github.com/Eulauss/divfree-FNS}}}

\hypersetup{
    colorlinks=true,
    urlcolor=blue,
}

\author[1, 2]{Juncai He\thanks{Email: jche@tsinghua.edu.cn}}
\author[3]{Xinliang Liu\thanks{Email: xinliang.liu@ouc.edu.cn}}
\author[1]{Zitong Tian\thanks{Email: tzt23@mails.tsinghua.edu.cn}}

\affil[1]{Qiuzhen College, Tsinghua University, Haidian District, Beijing 100084, China. }
\affil[2]{Yau Mathematical Sciences Center, Tsinghua University, Haidian District, Beijing 100084, China. }
\affil[3]{School of Mathematical Sciences, Ocean University of China, Qingdao, Shandong 266100, China.}

\newtheorem{theorem}{Theorem}[section]
\newtheorem{definition}[theorem]{Definition}
\newtheorem{remark}[theorem]{Remark}
\newtheorem{lemma}[theorem]{Lemma}

\newcommand{\FNS}{L_{n}^{k}}
\newcommand{\VFNS}{V_{n}^{k}}

\begin{document}

\date{}
\maketitle

\begin{abstract}
% Enforcing physical constraints, such as incompressibility (divergence-free conditions), is a fundamental challenge in numerical analysis. 
% % Traditional methods like the Finite Element Method (FEM) often require specialized elements and complex stabilization. 
% Classical mixed and divergence-conforming discretizations are highly effective on meshes. We target an explicit mesh-free approximation regime in which the divergence constraint is enforced exactly at the trial-space level.
% Recent advancements in Gaussian Processes (GPs) have shown that these constraints can be embedded directly into the function space via specialized, structure-preserving kernels. However, GP methods suffer from poor scalability. 

% We demonstrate how FNS arises as a scalable approximation of structure-preserving kernels.
% FNS can be viewed as a scalable approximation of structure-preserving kernels.

This paper studies the numerical approximation of divergence-free vector fields by linearized shallow neural networks, also referred to as random feature models or finite neuron spaces. Combining the stable potential lifting for divergence-free fields with the scalar Sobolev integral representation theory via ReLU$^k$ networks, we derive a core integral representation of divergence-free Sobolev vector fields through antisymmetric potentials parameterized by linearized ReLU$^k$ neural networks. This representation, together with a quasi-uniform distribution argument for the inner parameters, yields optimal approximation rates for such linearized ReLU$^k$ neural networks under an exact divergence-free constraint. Numerical experiments in two and three spatial dimensions, including $L^2$ projection and steady Stokes problems, confirm the theoretical rates and illustrate the effectiveness of exactly divergence-free conditions in computation.
\end{abstract}

\section{Introduction}
Many physical systems are governed by differential constraints that encode essential physical or geometric conservation laws. A prototypical example is the divergence-free condition
$$\nabla\cdot u = 0,$$
which is fundamental in numerous problems from scientific computing and engineering, including incompressible fluid dynamics, magnetic fields in Maxwell theory, and the coupled fields appearing in magnetohydrodynamics (MHD). In the numerical simulation of these systems, violating this constraint at the level of approximation (discretization) can seriously compromise both stability and physical fidelity. 
For incompressible flows, this issue is closely tied to the development of exactly divergence-free and pressure-robust discretizations; see, for example, \cite{guzman_conformingdivergencefree_2014,john2017divergence,neilan2020stokes}. For Maxwell equations, analogous structure-preserving ideas play a fundamental role in finite element and geometric discretizations of electromagnetic fields \cite{brenner2007locally,monk2019finite,campos2022variational}. In the context of magnetohydrodynamics (MHD), it has long been recognized that even small departures from the divergence-free condition may produce large numerical errors in computation \cite{brackbill1980effect,evans1988simulation,toth2000b,hu2017stable}. Furthermore, Brackbill and Barnes pointed out that violating this constraint at the discrete level gives rise to a spurious, physically meaningless force \cite{brackbill1980effect}.

The challenge of enforcing divergence constraints has driven decades of algorithmic innovation across multiple numerical paradigms. To contextualize the proposed method, we briefly review three dominant approaches: mesh-based discretizations, neural network solvers, and kernel methods.

%%%%%%%% Classical Mesh Methods %%%%%%%%%%%%

In \emph{classical mesh-based numerical analysis}, incompressibility is handled at the discrete level through several mature paradigms. Mixed formulations~\cite{brezzi_mixedhybridfinite_1991, boffi_mixedfiniteelementmethods_2013} introduce pressure as a Lagrange multiplier to impose $\nabla\cdot u=0$ weakly. This necessitates saddle-point linear algebra and requires the velocity--pressure approximation spaces to satisfy the delicate Ladyzhenskaya--Babu\v{s}ka--Brezzi (LBB) inf--sup compatibility condition. Alternatively, divergence-conforming discretizations~\cite{arnold_finiteelementexterior_2006, guzman_conformingdivergencefree_2014} utilize exact-sequence principles (e.g., Raviart--Thomas or N\'ed\'elec elements) to construct velocity spaces with the correct discrete divergence structure. Isogeometric B-spline approaches also successfully construct strictly divergence-free bases on structured patches~\cite{evans_isogeometricdivergenceconforming_2013}. Other common approaches include stabilization schemes (e.g., grad-div penalization)~\cite{franca_stabilizedfiniteelement_1993} and projection-based fractional-step methods~\cite{chorin_numericalsolutionnavierstokes_1968}. In these classical settings, the discrete enforcement of $\nabla\cdot u=0$ is well understood and highly effective. However, they rely heavily on specific mesh topologies, specialized elemental constructions, or staggered-grid structures, which can become highly restrictive in flexible, mesh-free, moving geometries, or high-dimensional settings.

%%%%%%%% Machine Learning Methods, both PINN and Linear %%%%%%%%%%%%

Recently, \emph{Scientific Machine Learning (SciML)} has emerged as a flexible, mesh-free alternative.
From a machine-learning viewpoint, unconstrained regressors may fit data or residuals well while still
violating fundamental invariants such as incompressibility. One important line of work, represented by
Physics-Informed Neural Networks (PINNs)~\cite{raissi_physics_2019, karniadakis_physics_2021},
typically imposes the divergence constraint through a soft penalty term in the loss function, e.g.,
$\lambda \|\nabla \cdot u\|_{L^2}^2$. This offers considerable flexibility, but can also introduce
nontrivial optimization and loss-balancing difficulties when the divergence penalty competes with data
fitting and PDE residual terms~\cite{wang_understanding_2021, krishnapriyan_characterizing_2021}.

A complementary line of work instead hard-encodes the constraint through divergence-free
parameterizations or structure-preserving neural solvers~\cite{richter-powell_neuralconservationlaws_2022,
cheng_decoupleddivergencefreeneural_2026, li_structurepreservingrandomizedneural_2026}, so that
incompressibility is satisfied exactly at the ansatz level. These works highlight the practical promise
of conservation-aware neural design and illustrate both general conservation-aware parameterizations and PDE-oriented solvers.
In contrast, our focus is on an explicit divergence-free approximation space together with
functional-analytic guarantees, including provable approximation rates that are consistent with the
observed numerical behavior.

Another established route to mathematically rigorous structure preservation is via \emph{operator-designed kernels in Gaussian process (GP) models}. Classical GP regression~\cite{rasmussen_gaussianprocesses_2006} provides a principled Bayesian nonparametric framework in which prior regularity and geometry are encoded by kernels. Its vector-valued extension allows physical constraints to be built directly into the hypothesis space: by applying differential operators to scalar potential kernels, one obtains matrix-valued kernels whose Reproducing Kernel Hilbert Space (RKHS) contains \emph{only} divergence-free or curl-free fields~\cite{alvarez_kernelsvectorvalued_2012, narcowich_divergencefreerbfs_2007, wendland_divergencefreekernelmethods_2009}. 

Progress on constrained GPs includes both general operator-constrained formulations and more specialized divergence-free models for hydrodynamics and vector fields~\cite{jidling_linearlyconstrainedgp_2017, owhadi_gaussianprocesshydrodynamics_2023, gia_vectorvaluedgaussianprocesses_2025}. Yet scalability remains the main bottleneck: exact GP inference requires dense covariance algebra with cubic cost, and even sparse or structured variants, including sparse GP approximations, inducing-variable and stochastic variational methods, KISS-GP, GPU implementations, and recent convergence analyses~\cite{quinonero_unifyingsparsegp_2005, titsias_variationalinducing_2009, hensman_scalablevariationalgpclassification_2015, wilson_kissgp_2015, gardner_gpytorch_2018, burt_convergencesparsevi_2020} remain costly and technically involved in constrained vector-valued settings. This motivates finite-dimensional, explicit-feature surrogates that retain the structure-preserving philosophy while replacing infinite-dimensional GP inference with scalable, convex linear algebra in a fixed feature space.

In this work, we develop a structure-preserving, mesh-free approximation framework that addresses two
limitations of the existing landscape: compared with constrained GP models, it avoids expensive
infinite-dimensional covariance algebra, and compared with recent conservation-aware neural solvers, it
aims to provide an explicit approximation space together with functional-analytic guarantees. For
arbitrary dimensions $d\ge 2$, we introduce a divergence-free finite neuron space (FNS) $\VFNS$ by applying
the operator $\mathcal{D}$ to skew-symmetric matrix-valued potentials whose entries are drawn from
scalar ReLU$^k$ finite neuron spaces. This yields an explicit finite-dimensional ansatz in which incompressibility is enforced exactly at the trial-space level while the outer coefficients are computed by
convex linear algebra.

By combining the stable potential lifting for divergence-free
fields on star-shaped Lipschitz domains from~\cite{costabel_bogovskiiregularizedpoincare_2010} with
the scalar Sobolev integral representation theory for ReLU$^k$ features
in~\cite{liu_integralrepresentationssobolev_2025}, we show that divergence-free Sobolev fields admit a
representation in this divergence-free feature system through antisymmetric potentials. 
This yields the key integral representation theorem of the paper; see Theorem~\ref{thm:Integral-representation-of-FNS}. Specifically, for $u \in H^r(\Omega, \mathbb{R}^d)$ with $\nabla \cdot u = 0$, we show that
\begin{equation*}
        u(x) = \int_{\mathbb{S}^{d}} \sum_{1 \le i < j \le d} \Phi_{\theta, ij}(x) \psi_{ij}(\theta) d\theta ~: ~\psi_{ij} \in L^{2}(\mathbb{S}^{d}),
\end{equation*}
    where $\Phi_{\theta, ij}(x) = \partial_j \sigma_k(\theta \cdot \tilde{x}) e_i - \partial_i \sigma_k(\theta \cdot \tilde{x}) e_j$ are the basis functions derived from the antisymmetric potential, with the augmented variable $\tilde x=(x,1)\in\mathbb{R}^{d+1}$, parameters
$\theta=(w,b)\in\mathbb{R}^{d+1}$ and ReLU$^k$ activation function $\sigma_k(\cdot)$.
This result forms the analytic bridge from the scalar FNS theory to the divergence-free setting and, under quasi-uniform sampling of the inner parameters, leads to the optimal approximation estimate
\[
\inf_{u_n\in \VFNS}\|u-u_n\|_{H^s(\Omega)}
\lesssim
n^{-(r-s)/d}\|u\|_{H^r(\Omega)},
\qquad 0\le s\le \min\{k-1,r\}.
\]
In this way, the optimal approximation rate for scalar linearized neural networks is carried over to divergence-free vector fields.

To assess both the theory and its practical value, we perform three main groups of numerical
experiments: divergence-free field approximation, steady Stokes problems with manufactured solutions,
and lid-driven cavity flow with nontrivial boundary data. Across these tests, in two and three spatial
dimensions, the observed error decay is broadly consistent with the theoretical predictions and
demonstrates that exact enforcement of incompressibility at the trial-space level leads to an effective and
practical structure-preserving discretization.

%%%%%%%%%%%% organization %%%%%%%%%%%
The remainder of the paper is organized as follows. Section~\ref{sec:preliminary} reviews scalar finite neuron space approximation, Sobolev integral representations, and the exterior-calculus-based potential framework for divergence-free fields. \Cref{sec:integral-representation,sec:error-rates} introduces the space $\VFNS$ and establishes the corresponding integral representation and approximation results. Section~\ref{sec:methods} describes the implementation, including quadrature, linear solvers, and the quasi-uniform generation of parameters. Section~\ref{sec:numerical-experiments} presents numerical experiments for $L^2$ approximation of divergence-free fields, the steady Stokes equations with homogeneous Dirichlet boundary conditions, and steady lid-driven cavity flow. Finally, Section~\ref{sec:conclusion} concludes with a brief discussion and several directions for future work. 
For brevity, additional tables and figures are deferred to Section~\ref{sec:extra-tables-figures}.

%%%%%%%%%%%%%%%%%%%%%%%%%%%%%%%%%%%%%%

 {}

\section{Preliminary}\label{sec:preliminary}

\subsection{Linearized networks}
Let $\sigma_k(t):=\mathrm{ReLU}^k(t)=(t)_+^k$ for an integer $k\ge 0$.
We use the augmented variable $\tilde x=(x,1)\in\mathbb{R}^{d+1}$ and parameters
$\theta=(w,b)\in\mathbb{R}^{d+1}$ so that we can write $\sigma_k(\theta\cdot\tilde x)=\sigma_k(w\cdot x+b)$.
By homogeneity of $\sigma_k$, it is standard to restrict $\theta$ to the unit sphere
$\mathbb{S}^d\subset\mathbb{R}^{d+1}$ by absorbing scale into the outer coefficient.

Previous work on \emph{nonlinear} approximation with ReLU-type dictionaries derives sharp convergence rates in suitable variation-type model classes~\cite{siegel_sharpboundsapproximation_2024}, and related greedy algorithms are known to realize these rates in many settings~\cite{mallat_matchingpursuitstimefrequency_1993,siegel_optimalconvergencerates_2022,siegel_greedytrainingalgorithms_2023}. 
In this paper we do not pursue nonlinear training. Instead, we adopt a \emph{linearized} strategy by fixing the inner parameters and solving only for the outer coefficients, which turns the approximation into a convex least-squares problem and is the regime used throughout our analysis and experiments.

We pre-fix inner parameters $\{\theta_j^\ast\}_{j=1}^n\subset\mathbb{S}^d$ and only optimize the
outer coefficients. Define the Finite Neuron Space
\[
    L_n^k
    :=\mathrm{span}\left\{\phi_{\ell}(x)\right\}_{\ell=1}^n,
    \qquad
    \phi_{\ell}(x):=\sigma_k(\theta_{\ell}^\ast\cdot\tilde x).
\]

\begin{theorem}[Integral representation of scalar Sobolev space, Theorem 2.3 of \cite{liu_integralrepresentationssobolev_2025}]\label{thm:integral-rep-scalar}
    The Sobolev space $H^{\frac{d+2k+1}{2}}(\Omega)$ is the RKHS associated with the ReLU$^k$ features. It can be represented as:
    \begin{equation}
        H^{\frac{d+2k+1}{2}}(\Omega)=\left\{\int_{\mathbb{S}^{d}}\sigma_{k}(\theta\cdot\tilde{x})\psi(\theta)d\theta:\psi\in L^{2}(\mathbb{S}^{d})\right\}.
    \end{equation}
    In particular, for any $f\in H^{\frac{d+2k+1}{2}}(\Omega)$, there exists a $\psi\in L^2(\mathbb{S}^d)$ (that may not be unique) such that
    \begin{equation}\label{eq:fpsi_int}
        f(x)=\int_{\mathbb{S}^d}\sigma_k(\theta\cdot\tilde x)\psi(\theta)\,d\theta.
    \end{equation}
    Furthermore, 
    \begin{equation}
        \|f\|_{H^{\frac{d+2k+1}{2}}(\Omega)}\simeq\inf\limits_{\psi\in L^2(\mathbb{S}^d)}\bigl\{\|\psi\|_{L^2(\mathbb{S}^d)}: \psi \mbox{ satisfies }\eqref{eq:fpsi_int}\bigr\}.
    \end{equation}
\end{theorem}
In practice, Sobolev spaces already cover a broad range of targets, and under a mild geometric condition on the pre-fixed inner parameters, one can recover the same approximation rates with this linear scheme.

Let $\rho(\cdot,\cdot)$ denote the geodesic distance on $\mathbb{S}^d$ and define the mesh norm
\[
h:=\max_{\theta\in\mathbb{S}^d}\min_{1\le j\le n}\rho(\theta,\theta_j^\ast).
\]
A set is \emph{well-distributed} if $h\lesssim n^{-1/d}$, and \emph{quasi-uniform} if additionally
$h$ is comparable to the separation distance $\min_{i\ne j}\rho(\theta_i^\ast,\theta_j^\ast)$.
These conditions ensure the neurons are spread evenly over the sphere and prevent clustering.
Then we have the linearized network approximation rates,
\begin{theorem}[Linear Approximation Rate, Theorem 2.2 in~\cite{liu_integralrepresentationssobolev_2025}]\label{thm:linear-approximation-rates}
    Let $\Omega\subset \mathbb{R}^d$ be a bounded domain with Lipschitz boundary. If $\{\theta_j^\ast\}_{j=1}^n$ is quasi-uniform, then for $f\in H^r(\Omega)$
    \[
        \inf_{f_n\in L_{n}^k}\|f-f_n\|_{H^s(\Omega)}
        \ \lesssim\ h^{\,r-s}\,\|f\|_{H^r(\Omega)}
        \ \lesssim\ n^{-(r-s)/d}\,\|f\|_{H^r(\Omega)},
        \qquad 0\le s\le \min\{k,r\},
    \]
    and in particular (taking $s=0$ and $r=\frac{d+2k+1}{2}$) one obtains the $L^2$ rate
    \[
        \inf_{f_n\in L_{n}^k}\|f-f_n\|_{L^2(\Omega)}
        \ \lesssim\
        n^{-\frac12-\frac{2k+1}{2d}}\ \|f\|_{H^{\frac{d+2k+1}{2}}(\Omega)}.
    \]
\end{theorem}
\begin{remark}
    Theorem~\ref{thm:linear-approximation-rates} is quoted here in a simplified form. The original result in~\cite{liu_integralrepresentationssobolev_2025} also includes coefficient control for the linear combination, which is not needed in the present work.
\end{remark}

\subsection{Divergence-free vector fields via exterior calculus and potentials}
\label{subsec:prelim_divfree_forms}

We briefly recall the exterior calculus~\cite{warner_foundationsdifferentiablemanifolds_1983,docarmo_differentialformsapplications_1994} viewpoint that unifies the classical $d=2$ stream-function
representation and the $d=3$ vector-potential representation, and extends them to general dimensions.

Throughout, we work in a bounded domain
$\Omega\subset\mathbb{R}^d$ with Lipschitz boundary and, when proving existence of potentials,
we assume in addition that $\Omega$ is starlike with respect to a ball. In particular, $[-1,1]^d$ satisfies all these conditions. For $0 \le m \le d$, let $\Lambda^m$ be the space of $m$-order differential forms on $\mathbb{R}^d$. Note that $\operatorname{dim}\Lambda^m = \operatorname{dim}\Lambda^{d-m} = \binom{d}{m}$.

Let $d:\Lambda^m\to\Lambda^{m+1}$ denote the exterior derivative and
$*:\Lambda^m\to\Lambda^{d-m}$ the Hodge star operator in Euclidean space.
Recall that
\[
    **\omega = (-1)^{m(d-m)}\,\omega \qquad \text{for all }\omega\in\Lambda^m,
\]
so applying $*$ twice returns the original form up to a sign.
All identities below are therefore understood up to this global sign, which does not affect our estimates.

We identify a vector field $u=(u_1,\dots,u_d)^\top:\Omega\to\mathbb{R}^d$ with the $1$-form
\[
    u \;:=\;\sum_{i=1}^d u_i\,dx_i .
\]
A direct computation shows
\begin{equation}\label{eq:div_as_dstar}
    d(*u) \;=\; (\nabla\cdot u)\,dx_1\wedge\cdots\wedge dx_d.
\end{equation}
This motivates the Sobolev divergence-free space
\[
    H^r_{\mathrm{div}}(\Omega)
    \;:=\;
    \Bigl\{
        u\in H^r(\Omega,\mathbb{R}^d):
        \nabla\cdot u = 0 \ \text{in the distributional sense}
    \Bigr\},
\]
equivalently $d(*u)=0$ in $\mathcal{D}'(\Omega)$ by \eqref{eq:div_as_dstar}.

The curl operator is generalised to general dimensions within the framework of differential forms.
For any $(d-2)$-form $\mu$ on $\Omega$, define generalized curl operator
\begin{equation}\label{eq:generalized_curl_def}
    \mathcal{D}\mu \;:=\; *\,d\mu,
\end{equation}
where $\mathcal{D}: \Lambda^{d-2}\to \Lambda^{1}$. Denote the corresponding vector field $u = \mathcal{D}\mu$.
Then $u$ is automatically divergence-free: since $d^2=0$,
\[
    d(*u)=d(**d\mu)=\pm d^2\mu =0,
\]
hence $\nabla\cdot u=0$ by \eqref{eq:div_as_dstar}. This provides a dimension-independent analogue of the
curl construction.

When $d=2$, a $(d-2)$-form is a $0$-form (scalar) $\mu$, and \eqref{eq:generalized_curl_def} reduces to the
stream-function representation:
\[
    u = *d\mu
    =
    \partial_{x_2}\mu\,dx_1 - \partial_{x_1}\mu\,dx_2
    \;\;\Longleftrightarrow\;\;
    u(x)=\bigl(\partial_{x_2}\mu(x),-\partial_{x_1}\mu(x)\bigr).
\]
When $d=3$, a $(d-2)$-form is a $1$-form (vector potential) $\mu=\sum_{i=1}^3 \mu_i\,dx_i$, and
$v=*d\mu$ coincides with the classical curl of the vector potential.

A convenient basis for $(d-2)$-forms is $\{\,*(dx_i\wedge dx_j)\,\}_{1\le i<j\le d}$.
We therefore represent the potential as
\begin{equation}\label{eq:mu_antisym_representation}
    \mu
    =
    \frac12\sum_{i,j=1}^d \mu_{ij}\,*(dx_i\wedge dx_j),
    \qquad \mu_{ji}=-\mu_{ij}.
\end{equation}
A direct calculation (see \cite{richter-powell_neuralconservationlaws_2022}, Appendix A) shows that, up to an overall sign
depending only on $d$ and the orientation convention,
\begin{equation}\label{eq:Dmu_rowwise_div}
    \mathcal{D}\mu
    =
    \sum_{i=1}^d
    \Bigl[
        \sum_{j=1}^d \partial_{x_j}\mu_{ij}
    \Bigr]\,dx_i .
\end{equation}
Equivalently, letting $(\mu)_{ij}\in\mathbb{R}^{d\times d}$ be a antisymmetric matrix field,
the vector field $\mathcal{D}\mu$ is the row-wise divergence,
\begin{align}
    u_i=(\mathcal{D}\mu)_i = \sum_{j=1}^d \partial_j\mu_{ij},\quad i=1,\dots,d. 
\end{align}

The representation $u=*d\mu$ is not only sufficient to be divergence-free but also essentially complete at Sobolev regularity on starlike domains. We will use the following lemma from~\cite{costabel_bogovskiiregularizedpoincare_2010}.

\begin{lemma}[Potential representation and $H^{r+1}$-stability]\label{lem:potential_stable}
Assume $\Omega\subset\mathbb{R}^d$ is bounded and starlike with respect to a ball.
Let $r\ge 0$ and let $u\in H^r_{\mathrm{div}}(\Omega)$.
Then there exists a $(d-2)$-form $\mu\in H^{r+1}(\Omega,\Lambda^{d-2})$ such that
\begin{equation}\label{eq:surj_D}
    u = \mathcal{D}\mu = *\,d\mu \quad\text{in }\Omega,
\end{equation}
and moreover
\begin{equation}\label{eq:stable_potential_bound}
    \|\mu\|_{H^{r+1}(\Omega)}
    \;\le\;
    C_\Omega\,\|u\|_{H^r(\Omega)} .
\end{equation}
\end{lemma}
\begin{proof}
Set $v:=*u\in H^r(\Omega,\Lambda^{d-1})$. Since $u\in H^r_{\mathrm{div}}(\Omega)$, we have $d v = d(*u)=0$
by \eqref{eq:div_as_dstar}. Applying \cite[Prop.~4.1(i)]{costabel_bogovskiiregularizedpoincare_2010} with $\ell=d-1$ and $s=r$,
there exists $\tilde\mu\in H^{r+1}(\Omega,\Lambda^{d-2})$ such that
\begin{align}
    d\tilde\mu =v, \quad
    \|\tilde\mu\|_{H^{r+1}} \le C_\Omega\|v\|_{H^r}.
\end{align}
Finally, $u$ is recovered as $u=\pm *v = \pm *d\tilde\mu$. Absorbing the sign into the definition of $\mu$ yields \eqref{eq:surj_D} and
\eqref{eq:stable_potential_bound}.
\end{proof}

\section{Integral Representation of Divergence-Free Vector Fields via \texorpdfstring{$\mathrm{ReLU}^k$}{ReLUk} Features}\label{sec:integral-representation}

We construct and analyze divergence-free finite neuron spaces using the $\mathrm{ReLU}^k$ activation
$\sigma_k(t)=\max(0,t)^k$ in dimension $d\ge2$. Throughout we assume $k\ge1$ so that the required derivatives are well-defined.

Our analysis connects approximation of finite neuron spaces in Sobolev spaces and the RKHS viewpoint for ReLU$^k$ networks~\cite{liu_integralrepresentationssobolev_2025, Liu2025NTK}. We work on a bounded Lipschitz domain $\Omega\subset\mathbb{R}^d$ that is starlike with respect to a ball~\cite{costabel_bogovskiiregularizedpoincare_2010}.

We prove an integral representation theorem for $H^r_{\mathrm{div}}(\Omega)$.

\begin{theorem}[Integral representation of $H^r_{\mathrm{div}}(\Omega)$]
\label{thm:Integral-representation-of-FNS}
Let $k\ge 1$, and let $\Omega \subset \mathbb{R}^d$ be a bounded domain with Lipschitz boundary and starlike with respect to a ball.
Set $r=\frac{d+2k-1}{2}$.
Then for every $u\in H^r_{\mathrm{div}}(\Omega)$, there exist densities $\psi_{ij}\in L^2(\mathbb{S}^d), \  1\le i<j\le d,$
such that
\[
u(x)=\int_{\mathbb{S}^d}\sum_{1\le i<j\le d}\Phi_{\theta,ij}(x)\psi_{ij}(\theta)\,d\theta,
\]
where $\Phi_{\theta,ij}(x) = \partial_j\sigma_k(\theta\cdot \tilde x)e_i - \partial_i\sigma_k(\theta\cdot \tilde x)e_j$.
Moreover,
\[
\|u\|_{H^r(\Omega)}
\simeq
\inf\left\{
\left(
\sum_{1\le i<j\le d}\|\psi_{ij}\|_{L^2(\mathbb{S}^d)}^2
\right)^{1/2}
:\ 
u=\int_{\mathbb{S}^d}\sum_{1\le i<j\le d}\Phi_{\theta,ij}(x)\psi_{ij}(\theta)\,d\theta
\right\}.
\]
In particular, one may choose the representing family so that
\[
\left(
\sum_{1\le i<j\le d}\|\psi_{ij}\|_{L^2(\mathbb{S}^d)}^2
\right)^{1/2}
\lesssim
\|u\|_{H^r(\Omega)}.
\]
\end{theorem}

\begin{proof}
The proof relies on the isomorphism established in Lemma~\ref{lem:potential_stable} and the scalar integral representation.

First, recall the smoothness relationship: $r = \frac{d+2k-1}{2}$ implies $r+1 = \frac{d+2k+1}{2}$. By Theorem \ref{thm:integral-rep-scalar}, the scalar Sobolev space $H^{r+1}(\Omega)$ admits the integral representation using ReLU$^k$ features.

Let $u \in H^r_{\mathrm{div}}(\Omega)$. By Lemma~\ref{lem:potential_stable}, there exists a potential $\mu \in H^{r+1}(\Omega,\Lambda^{d-2})$ such that $u = \mathcal{D}\mu$. Fixing the basis \(\{*(dx_i\wedge dx_j)\}_{i<j}\) of \(\Lambda^{d-2}\), the Sobolev norm of \(\mu\) is equivalent to the Euclidean sum of the Sobolev norms of its scalar coefficients \(\mu_{ij}\). Then, for each component $\mu_{ij} \in H^{r+1}(\Omega)$, Theorem~\ref{thm:integral-rep-scalar} guaranties the existence of densities $\psi_{ij} \in L^2(\mathbb{S}^d)$ such that:
\begin{equation}
    \mu_{ij}(x) = \int_{\mathbb{S}^{d}}\sigma_{k}(\theta\cdot\tilde{x})\psi_{ij}(\theta)d\theta.
\end{equation}
To ensure $\mu$ is antisymmetric, we enforce $\psi_{ji} = -\psi_{ij}$ (and $\psi_{ii} = 0$).
Substituting this into the divergence formula:
\begin{equation}
    u(x) = \mathcal{D}\mu(x) \implies u_p(x) = \sum_{q=1}^d \partial_q \mu_{pq}(x) = \sum_{q=1}^d \int_{\mathbb{S}^{d}} \partial_q \sigma_{k}(\theta\cdot\tilde{x})\psi_{pq}(\theta) d\theta.
\end{equation}
We rewrite the sum by grouping terms for each unique pair $\{i,j\}$ with $1 \le i < j \le d$. The contribution of the pair $(i,j)$ to the vector field involves terms with $\psi_{ij}$ and $\psi_{ji}$:
\begin{equation}
    \partial_j \sigma_k \psi_{ij} e_i + \partial_i \sigma_k \psi_{ji} e_j = (\partial_j \sigma_k e_i - \partial_i \sigma_k e_j) \psi_{ij} = \Phi_{\theta, ij}(x) \psi_{ij}(\theta).
\end{equation}
Summing over all pairs yields the stated integral representation.

For the coefficient bound, write
\[
\mu=\sum_{1\le i<j\le d}\mu_{ij}\,*(dx_i\wedge dx_j).
\]
By Theorem~\ref{thm:integral-rep-scalar}, the densities may be chosen so that
\[
\sum_{1\le i<j\le d}\|\psi_{ij}\|_{L^2(\mathbb{S}^d)}^2
\lesssim
\sum_{1\le i<j\le d}\|\mu_{ij}\|_{H^{r+1}(\Omega)}^2
\simeq
\|\mu\|_{H^{r+1}(\Omega)}^2
\lesssim
\|u\|_{H^r(\Omega)}^2.
\]
Conversely, let $\{\psi_{ij}\}_{1\le i<j\le d}\subset L^2(\mathbb{S}^d)$ be any family defining the above representation, and define
\[
\mu_{ij}(x):=\int_{\mathbb{S}^d}\sigma_k(\theta\cdot\tilde{x})\psi_{ij}(\theta)\,d\theta,
\qquad
\mu_{ji}:=-\mu_{ij},\quad \mu_{ii}:=0.
\]
Then $\mu\in H^{r+1}(\Omega,\Lambda^{d-2})$, and by the same computation as above, the represented field is exactly $u=\mathcal D\mu$.
Moreover,
\[
\|\mu\|_{H^{r+1}(\Omega)}^2
\lesssim
\sum_{1\le i<j\le d}\|\psi_{ij}\|_{L^2(\mathbb{S}^d)}^2.
\]
Since $\mathcal D:H^{r+1}(\Omega,\Lambda^{d-2})\to H^r(\Omega,\Lambda^1)$ is continuous, it follows that
\[
\|u\|_{H^r(\Omega)}
\lesssim
\left(
\sum_{1\le i<j\le d}\|\psi_{ij}\|_{L^2(\mathbb{S}^d)}^2
\right)^{1/2}.
\]
This proves the norm characterization.
\end{proof}

The integral representation in Theorem~\ref{thm:Integral-representation-of-FNS} is the key structural result behind the divergence-free finite neuron space. It identifies the continuous feature dictionary underlying $V_n^k$ and shows that every field in $H_{\mathrm{div}}^r(\Omega)$ can be generated through antisymmetric potentials built from ReLU$^k$ features. In particular, the divergence-free constraint is encoded directly at the level of the representation, rather than being imposed afterward by a penalty term or a Lagrange multiplier. This is important for two reasons. First, it provides the analytic bridge from the scalar approximation theory to the divergence-free setting, since approximation of the potential coefficients by scalar finite neuron spaces can be transferred to approximation of the vector field through the continuous operator $\mathcal{D}$. Second, it shows that the discrete space $V_n^k$ should be viewed as a finite-dimensional discretization of this continuous representation; consequently, once the inner parameters are chosen quasi-uniformly, the optimal approximation rates from the scalar theory can be inherited by the divergence-free space. In this sense, Theorem~\ref{thm:Integral-representation-of-FNS} is the mechanism that connects exact structure preservation with sharp approximation estimates.

\section{Divergence-Free Finite Neuron Spaces and Optimal Error Rates}\label{sec:error-rates}
\subsection{Divergence-free finite neuron spaces}

\begin{definition}[Divergence-free finite neuron space]
Let $\{\theta_{\ell}^{*}\}_{\ell=1}^{n}\subset\mathbb{S}^{d}$ be a quasi-uniform set of parameters, and define the scalar finite neuron space (FNS)
\[
    \FNS := \mathrm{span}\{\phi_{\ell}\}_{\ell=1}^{n},
    \qquad 
    \phi_{\ell}(x):=\sigma_{k}(\theta_{\ell}^{*}\cdot\tilde{x}),
\]
where $\tilde{x}:=(x,1)\in\mathbb{R}^{d+1}$ and $\theta_\ell^*=(w_\ell^*,b_\ell^*)\in\mathbb{R}^{d}\times \mathbb{R}$.
A \emph{potential matrix} is a skew-symmetric matrix field $\mu(x)\in\mathbb{R}^{d\times d}$ whose entries satisfy
\[
    \mu_{ij}\in\FNS \ \ (i<j),\qquad \mu_{ji}=-\mu_{ij},\qquad \mu_{ii}=0.
\]
We define the divergence-free finite neuron space as the image of such potentials under the differential operator $\mathcal{D}$:
\[
    \VFNS
    :=\Bigl\{\, \mathcal{D}\mu \ :\ \mu \text{ is skew-symmetric with } \mu_{ij}\in\FNS \ (i<j)\Bigr\}.
\]
\end{definition}

Equivalently, writing $\mu_{ij}(x)=\sum_{\ell=1}^n a_{\ell,ij}\phi_\ell(x)$ for $i<j$ yields
$n\cdot \binom{d}{2}$ outer linear parameters $\{a_{\ell,ij}\}$.
For each $\ell$ and $1\le i<j\le d$, the coefficient $a_{\ell,ij}$ multiplies the vector-valued basis function
$\Phi_{\ell,ij}(x):=\mathcal{D}\mu(x)$ generated by setting $\mu_{ij}=\phi_\ell$, $\mu_{ji}=-\phi_\ell$, and all other entries zero. In particular,
\begin{align}\label{eq:basis-of-FNS}
    \Phi_{\ell,ij}(x)
    &= \partial_j \phi_{\ell}(x)\, e_i - \partial_i \phi_{\ell}(x)\, e_j \\
    &= k\,\sigma_{k-1}(\theta_{\ell}^{*}\cdot\tilde{x})\,
       \bigl(w_{\ell,j}^* e_i - w_{\ell,i}^* e_j\bigr),
\nonumber
\end{align}
and $\VFNS=\mathrm{span}\{\Phi_{\ell,ij}:\ \ell=1,\dots,n,\ 1\le i<j\le d\}$.

\subsection{Convergence Analysis}
We leverage the optimal approximation rates established for the scalar FNS when using quasi-uniform points, which effectively serve as an optimal quadrature rule for the integral representation.

\begin{theorem}[Convergence of $\VFNS$ for General Dimension $d \ge 2$]\label{thm:approximation-rates-of-div-free-FNS}
    Let $\Omega \subset \mathbb{R}^d$ be a bounded domain with Lipschitz boundary and starlike with respect to a ball. Let $u \in H^r_{\mathrm{div}}(\Omega)$. Assume $k \ge 1$, $r+1 \le \frac{d+2k+1}{2}$ and $0\le s \le \min\{k-1, r\}$. Let $\VFNS$ be the Divergence-Free FNS constructed with quasi-uniform parameters. 
    such that
    \begin{equation}
        \inf_{u_n\in V_{n}^{k}}\|u-u_n\|_{H^s(\Omega)}
    \lesssim
    n^{-\frac{r-s}{d}}\|u\|_{H^r(\Omega)}.
    \end{equation}

\end{theorem}

\begin{proof}
Since $u \in H^r_{\mathrm{div}}(\Omega)$, by Lemma~\ref{lem:potential_stable}, there exists an antisymmetric tensor field $\mu = (\mu_{ij})_{i,j=1}^d$ with components $\mu_{ij} \in H^{r+1}(\Omega)$ such that $u$ is the divergence of $\mu$:
\begin{equation}
    u_i(x) = \sum_{j=1}^d \frac{\partial \mu_{ij}(x)}{\partial x_j},
\end{equation}
 and
\[
\sum_{1\le i<j\le d}\|\mu_{ij}\|_{H^{r+1}(\Omega)}^2
\simeq
\|\mu\|_{H^{r+1}(\Omega)}^2
\lesssim
\|u\|_{H^r(\Omega)}^2.
\]

We apply the approximation results for scalar FNS (Theorem~\ref{thm:linear-approximation-rates}) to each component $\mu_{ij}$ of the potential. For each $\mu_{ij} \in H^{r+1}(\Omega), i<j$, there exists an approximation $\tilde{\mu}_{ij} \in \FNS$ (constructed using quasi-uniform parameters) such that:
\begin{equation}
    \|\mu_{ij} - \tilde{\mu}_{ij}\|_{H^{1+s}(\Omega)} \lesssim n^{-\frac{(r+1)-1-s}{d}}\|\mu_{ij}\|_{H^{r+1}(\Omega)} = n^{-\frac{r-s}{d}}\|\mu_{ij}\|_{H^{r+1}(\Omega)}.
\end{equation}
Then we extend to all $1\le i,j\le d$ by $\tilde\mu_{ji}:=-\tilde\mu_{ij}$ and $\tilde\mu_{ii}=0$ to keep $\tilde{\mu}$ antisymmetric.
We construct the approximate vector field $v_n$ using these approximate potentials:
\begin{equation}
    (u_n)_i = \sum_{j=1}^d \frac{\partial \tilde{\mu}_{ij}}{\partial x_j}.
\end{equation}
By linearity, $u_n \in \VFNS$. Since $\mathcal D=*d$ is continuous from $H^{1+s}(\Omega,\Lambda^{d-2})$ to $H^s(\Omega,\Lambda^1)$, we have
\begin{align}
    \|u - u_n\|_{H^{s}(\Omega)}
    &\lesssim \sum_{i,j=1}^d \|\mu_{ij} - \tilde{\mu}_{ij}\|_{H^{1+s}(\Omega)} \\
    &\lesssim n^{-\frac{r-s}{d}} \sum_{i,j=1}^d \|\mu_{ij}\|_{H^{r+1}(\Omega)}
    \lesssim n^{-\frac{r-s}{d}}\|u\|_{H^r(\Omega)}.
\end{align}
\end{proof}

\begin{remark}
Measured in degrees of freedom, standard finite element rates necessarily carry a factor $1/d$ in the exponent. Indeed, on a uniform mesh in dimension $d$, one has $N\sim h^{-d}$, so for a velocity approximation of order $p$, the estimates $O(h^{p-1})$ in $H^1$ and $O(h^{p})$ in $L^2$ become $O(N^{-(p-1)/d})$ and $O(N^{-p/d})$, respectively. By contrast, Theorem~\ref{thm:approximation-rates-of-div-free-FNS} yields divergence-free FNS rates with a dimension-independent leading term $1/2$, so the guaranteed convergence is always at least $O(n^{-1/2})$. In particular, when $d=2$, the $L^2$ rate of the Taylor--Hood $P_2/P_1$ method is $N^{-3/2}$, whereas the $k=3$ divergence-free FNS yields the sharper rate $n^{-7/4}$. The numerical results in Section~\ref{subsubsec:regularized-lid} are consistent with this comparison.
\end{remark}

\begin{remark}[General domains and harmonic fields]
Lemma~\ref{lem:potential_stable} uses that $\Omega$ is starlike (hence topologically trivial), so that every closed $(d-1)$-form is exact. On a general bounded Lipschitz domain, a divergence-free field admits instead a Hodge-type decomposition
\[
  u = \mathcal D\mu + h,
\]
where $h$ belongs to a finite-dimensional harmonic subspace determined by the de~Rham cohomology of $\Omega$. Accordingly, one may extend the present divergence-free FNS by adjoining a basis of harmonic fields and solve for the potential and harmonic coefficients simultaneously, see \cite{arnold_finiteelementexterior_2006,costabel_bogovskiiregularizedpoincare_2010}.
\end{remark}
% =========================
\section{Methods}\label{sec:methods}
% =========================
In this section, we consider $d\in\{2,3\}$. Let $\Omega=[-1,1]^d$, a bounded domain with Lipschitz boundary and starlike with respect to a ball.

\subsection{Divergence-free vector field \texorpdfstring{$L^2$}{L2} approximation}
Given a divergence-free target vector field
$\mathbf{u}^\dagger:\Omega\to\mathbb{R}^d$ with $\nabla\cdot \mathbf{u}^\dagger=0$,
we approximate it in the divergence-free FNS $\VFNS$.
We solve the following projection problem:
\begin{equation}
  \hat{\mathbf{u}}_n
  \;=\;
  \arg\min_{\mathbf{v}\in \VFNS}
  \;\|\mathbf{v}-\mathbf{u}^\dagger\|_{L^2(\Omega)}^2
  \;=\;
  \arg\min_{\mathbf{v}\in \VFNS}
  \int_{\Omega} |\mathbf{v}(x)-\mathbf{u}^\dagger(x)|^2\,dx.
  \label{eq:l2_projection}
\end{equation}
Since $\VFNS$ is finite-dimensional and linear, \eqref{eq:l2_projection} reduces to a convex quadratic problem.

\subsection{Solving Stokes Equation in divergence-free FNS}

A main advantage of divergence-free FNS is that the incompressibility constraint is built into the trial space. This makes it natural to work with the reduced velocity formulation of incompressible problems, in which the pressure does not need to be approximated explicitly. As an example, we consider the steady Stokes problem on $\Omega=[-1,1]^d$ with homogeneous Dirichlet boundary conditions:
\begin{align}
    \begin{cases}
        -\nu \Delta \mathbf{u} + \nabla p = \mathbf{f} \qquad \text{in } \Omega, \\
        \nabla \cdot \mathbf{u} = 0 \qquad \text{in } \Omega, \\
        \mathbf{u} = \mathbf{0} \qquad \text{on } \partial\Omega .
    \end{cases}
\end{align}

Let
\[
  H^1_{\mathrm{div}}(\Omega)
  :=
  \{\mathbf{v}\in H^1(\Omega,\mathbb{R}^d): \nabla\cdot \mathbf{v}=0\},
\]
and define the zero-trace divergence-free space
\[
  V := H^1_{\mathrm{div}}(\Omega)\cap H_0^1(\Omega,\mathbb{R}^d).
\]
The standard mixed weak formulation seeks $(\mathbf{u},p)\in H_0^1(\Omega,\mathbb{R}^d)\times L_0^2(\Omega)$ such that
\[
  \nu\int_\Omega \nabla \mathbf{u}:\nabla \mathbf{v}\,dx
  - \int_\Omega p\,\nabla\cdot \mathbf{v}\,dx
  = \int_\Omega \mathbf{f}\cdot \mathbf{v}\,dx
  \qquad
  \forall\,\mathbf{v}\in H_0^1(\Omega,\mathbb{R}^d),
\]
together with
\[
  \int_\Omega q\,\nabla\cdot \mathbf{u}\,dx = 0
  \qquad
  \forall\,q\in L_0^2(\Omega).
\]
If one restricts the test space to $V$, then $\nabla\cdot \mathbf{v}=0$ by construction, so the pressure term vanishes identically. Hence the velocity component is characterized by the reduced variational problem
\[
  \nu\int_\Omega \nabla \mathbf{u}:\nabla \mathbf{v}\,dx
  =
  \int_\Omega \mathbf{f}\cdot \mathbf{v}\,dx
  \qquad
  \forall\,\mathbf{v}\in V.
\]
Equivalently, $\mathbf{u}$ is the unique minimizer of the reduced energy
\[
  \frac{\nu}{2}\int_\Omega |\nabla \mathbf{v}|^2\,dx
  - \int_\Omega \mathbf{f}\cdot \mathbf{v}\,dx
\]
over $V$. In this sense, the pressure is not discarded from the continuous Stokes model; rather, it acts as the Lagrange multiplier for the incompressibility constraint in the mixed formulation, and disappears after passing to the divergence-free reduced formulation.

In our implementation, the divergence constraint is enforced exactly by the trial space, while the homogeneous boundary condition is imposed weakly by a penalty term. Given viscosity $\nu>0$, forcing $\mathbf{f}$, and a penalty parameter $\varepsilon>0$, we minimize
\begin{equation}
  \mathbf{u}_\varepsilon
  \;=\;
  \arg\min_{\mathbf{v}\in H^1_{\mathrm{div}}(\Omega)}
  J_\varepsilon(\mathbf{v}),
  \qquad
  J_\varepsilon(\mathbf{v})
  :=
  \frac{\nu}{2}\int_\Omega |\nabla\mathbf{v}|^2\,dx
  -\int_\Omega \mathbf{f}\cdot \mathbf{v}\,dx
  +\frac{1}{2\varepsilon}\int_{\partial\Omega} |\mathbf{v}|^2\,ds.
  \label{eq:stokes_energy}
\end{equation}
This is a penalized version of the reduced divergence-free formulation above: incompressibility is enforced exactly through the admissible space, whereas the zero boundary condition is enforced softly through the boundary term.

We then discretize $H^1_{\mathrm{div}}(\Omega)$ by a divergence-free finite neuron space $\VFNS\subset H^1_{\mathrm{div}}(\Omega)$ and compute
\begin{equation}
  \hat{\mathbf{u}}_n
  \;=\;
  \arg\min_{\mathbf{v}\in \VFNS} J_\varepsilon(\mathbf{v}),
  \label{eq:stokes_fns_min}
\end{equation}
which again yields a convex quadratic problem in the linear coefficients. Since every function in $\VFNS$ is divergence-free by construction, no separate pressure variable or discrete inf-sup compatibility condition is required.

\subsection{Gauss-Legendre quadrature for Inner products}
We approximate all inner products over $\Omega$ and $\partial\Omega$ by deterministic quadrature.

\paragraph{Integrals on $\Omega=[-1,1]^d$.}
We partition each coordinate interval into $N_x$ subintervals of size $h=2/N_x$,
forming $N_x^d$ subcubes. On each subcube, we apply a tensor-product Gauss--Legendre rule of
order \texttt{order} in each dimension. For a generic integrand $g$,
\[
  \int_\Omega g(x)\,dx
  \;\approx\;
  \sum_{q=1}^{N_q} w_q\,g(x_q),
  \qquad
  N_q = N_x^d \cdot (\texttt{order})^d,
\]
where $(x_q,w_q)$ are the piecewise-tensor Gauss nodes and weights mapped to physical coordinates. Inspired by~\cite{mao_solvinghighdimensionalpdes_2026}, we choose $N_x=200, \texttt{order}=5$ and $N_x=40, \texttt{order}=3$ for $d=2,3$ respectively.
To control peak memory usage, we evaluate $g(x_q)$ in mini-batches (chunks) and accumulate the quadrature sums incrementally.

\paragraph{Boundary integrals on $\partial\Omega$.}
The boundary $\partial\Omega$ consists of $2d$ faces $\{x_i=\pm 1\}\cong [-1,1]^{d-1}$.
We apply the same piecewise tensor Gauss--Legendre rule on each face (in dimension $d-1$) and sum
the contributions:
\[
  \int_{\partial\Omega} h(x)\,ds
  \;\approx\;
  \sum_{\text{faces }F}\ \sum_{m} \omega_{F,m}\, h(x_{F,m}).
\]
This is used for the penalty term $\int_{\partial\Omega}|\mathbf{u}_n|^2\,ds$.

\subsection{Solving the induced least square problem}
Given $n$ inner parameters $\{\theta^*_\ell\}_{\ell=1}^n$, define
\[
  \VFNS := \mathrm{span}\big\{\Phi_{\ell,ij}:\ \ell=1,\dots,n,\ 1\le i<j\le d\big\},
  \qquad
  P := \dim(\VFNS) = n\binom{d}{2}.
\]
We parameterize $\mathbf{u}_n\in \VFNS$ as
\[
  \mathbf{u}_n(x)=\sum_{\ell=1}^n\sum_{1\le i<j\le d} a_{\ell,ij}\,\Phi_{\ell,ij}(x)
  \;=\;
  \sum_{p=1}^{P} a_p\,\Phi_p(x),
\]
where $\{\Phi_p\}_{p=1}^P$ is any fixed ordering of the basis functions.

Let $\{(x_q,\omega_q)\}_{q=1}^{N_q}$ be the Gauss--Legendre quadrature nodes and weights for volume and boundary integrals. We use the same notation for simplicity.
Both the $L^2$ projection and the penalized Stokes energy minimization reduce to a weighted least-squares problem of the form
\begin{equation}\label{eq:weighted_ls}
  \hat a
  \;=\;
  \arg\min_{a\in\mathbb{R}^P}\ \|W(Ha-y)\|_2^2,
\end{equation}
where $H\in\mathbb{R}^{N_q\times P}$ is a feature matrix evaluated on quadrature nodes, $y\in\mathbb{R}^{N_q}$ is the corresponding target vector, and $W$ is a diagonal weight matrix typically involving $\sqrt{\omega_q}$
(and also $\sqrt{\nu}$, $\sqrt{1/\varepsilon}$ for Stokes).
We have two basic strategies to solve this least-squares problem~\cite{golub_matrixcomputations_2013}.
\subsubsection{Normal-equation approach}
A classic strategy is to assemble the normal system
\begin{equation}\label{eq:normal_eq}
  A\,\hat a=b,
  \qquad
  A:=H^\top W^\top W H\in\mathbb{R}^{P\times P},
  \quad
  b:=H^\top W^\top W y\in\mathbb{R}^{P}.
\end{equation}
This corresponds to assembling the mass matrix of basis functions under the relevant inner product and yields a semi-positive definite system.
The main practical advantage is that the final linear system has size only $P\times P$, independent of the number of quadrature nodes $N$.
In our implementation, $A$ and $b$ are assembled by streaming over quadrature nodes in chunks to control peak memory usage.

However, forming the normal equations squares the condition number in the $2$-norm: roughly, $\kappa_2(A)\approx \kappa_2(W H)^2$.
Thus, when $WH$ is already moderately ill-conditioned, $A$ may become numerically hard to solve accurately,
even though it is much smaller. We conduct a comparison in Section~\ref{subsubsec:mass_ablation} and Figure~\ref{fig:mass-ablation}.

\subsubsection{Direct least-squares approach without forming mass matrix}
An alternative is to solve the weighted least-squares problem \eqref{eq:weighted_ls} directly,
instead of forming $A=H^\top W^\top W H$.
Define $\tilde H := W H$ and $\tilde y := W y$, then \eqref{eq:weighted_ls} becomes
\[
  \hat a=\arg\min_a \|\tilde H a-\tilde y\|_2^2.
\]
A numerically robust route is to compute an SVD (or QR) factorization of $\tilde H$.
With SVD $\tilde H=U\Sigma V^\top$, the minimum-norm solution is
\begin{equation}\label{eq:svd_pinv}
  \hat a = \tilde H^+ \tilde y = V\,\Sigma^+\,U^\top \tilde y,
\end{equation}
where $(\cdot)^+$ is the Moore–Penrose pseudoinverse, i.e. $\Sigma^+ = \mathrm{diag}(1/\sigma_i)$ on the support of $\sigma$ and $0$ elsewhere.
Compared with normal equations, this avoids squaring the condition number and often produces more accurate solutions when the Gram system becomes difficult to solve.

The price is that one typically needs access to the full tall matrix $\tilde H$ (size $N_q\times P$),
whose row dimension $N_q$ is proportional to the number of quadrature nodes.
Thus, direct least squares can be much more memory demanding than assembling $A$.
When $\kappa_2(WH)$ is large, solving the normal equations may suffer from amplified roundoff errors; in such regimes, a direct least-squares solver often yields noticeably improved accuracy, at the expense of storing or factorizing the tall matrix.

%%%%%%%%%%%%%%%%%%%%%%%%%%

\subsection{Quasi-uniform sampling of neuron parameters on \texorpdfstring{$\mathbb{S}^d$}{Sd}}
The neuron parameters $\theta_\ell=(w_\ell,b_\ell)\in\mathbb{S}^d$ should be well-spread on the sphere
to reduce clustering and improve conditioning. We use a two-stage strategy:
\begin{enumerate}
  \item \textbf{Gaussian initialization.}
  Sample $g_\ell\sim\mathcal{N}(0,I_{d+1})$ i.i.d. and set
  $\theta_\ell^{(0)} = g_\ell/\|g_\ell\|_2$, which is uniform on $\mathbb{S}^d$.
  \item \textbf{Energy refinement.}
  Starting from $\{\theta_\ell^{(0)}\}$, we minimize a repulsive Riesz energy
  \[
    E_s(\theta_1,\dots,\theta_n)
    :=
    \sum_{1\le i<j\le n}\frac{1}{\|\theta_i-\theta_j\|_2^{s}},
    \qquad s=d-1,
  \]
  via projected gradient steps, with projection $\theta_\ell\leftarrow \theta_\ell/\|\theta_\ell\|_2$
  after each update. We use $\texttt{jaxopt.LBFGS}$ to optimize the above energy in implementation. The resulting set is empirically quasi-uniform on $\mathbb{S}^d$.
  
\end{enumerate}

% =========================
\section{Numerical Experiments}\label{sec:numerical-experiments}
% =========================

All experiments are carried out in double precision (\texttt{float64}) for numerical stability.
All linear systems (both normal equations and direct least-squares formulations) are solved by
\texttt{scipy.linalg.lstsq} using the SVD-based driver \texttt{gelsd}. In our implementation, we discard neurons whose hyperplanes do not intersect the domain $\Omega$. 

% -----------------------------------------
\subsection{Divergence-free vector field \texorpdfstring{$L^2$}{L2} approximation}
% -----------------------------------------

\paragraph{Setup.}
We consider $\Omega=[-1,1]^d$ with $d\in\{2,3\}$ and test activation powers
\[
k\in\{1,2,3,4\}\qquad \text{(for both $d=2$ and $d=3$)}.
\]
For each $(d,k)$, we sample $n$ quasi-uniform parameters $\{\theta_\ell\}_{\ell=1}^n\subset\mathbb{S}^d$
and report errors vs.\ $n$.
We solve the $L^2(\Omega)$ projection:
\[
\mathbf{u}_n
=
\arg\min_{\mathbf{v}\in \VFNS}\|\mathbf{v}-\mathbf{u}^\dagger\|_{L^2(\Omega)}^2.
\]
As shown in~\cite{liu_stabilityshallowneural_2025}, the mass matrix becomes increasingly ill-conditioned as $d$ decreases for fixed $k$. Accordingly, we solve the weighted least-squares system directly in $d=2$ for numerical stability, while using the normal equations in $d=3$ for efficiency.

\paragraph{Training objective.}
We use explicit divergence-free targets with a tunable frequency parameter $\omega$
(default $\omega=\pi$ in our experiments).

\textbf{2D target (stream function).}
Let
\[
\psi(x,y) = \sin(\omega x)\sin(\omega y),
\qquad
\mathbf{u}^\dagger_{2D}=\nabla^\perp\psi
=
\begin{pmatrix}
\partial_y\psi\\
-\partial_x\psi
\end{pmatrix}.
\]
Equivalently,
\[
\mathbf{u}^\dagger_{2D}(x,y)
=
\omega
\begin{pmatrix}
\sin(\omega x)\cos(\omega y)\\
-\cos(\omega x)\sin(\omega y)
\end{pmatrix}.
\]

\textbf{3D target (curl of a vector potential).}
Define
\[
\mathbf{A}(x,y,z)=
\begin{pmatrix}
\sin(\omega y)\sin(\omega z)\\
\sin(\omega z)\sin(\omega x)\\
\sin(\omega x)\sin(\omega y)
\end{pmatrix},
\qquad
\mathbf{u}^\dagger_{3D}=\nabla\times\mathbf{A}.
\]
Writing components explicitly,
\[
\begin{aligned}
u^\dagger_1 &= \omega\,\sin(\omega x)\big(\cos(\omega y)-\cos(\omega z)\big),\\
u^\dagger_2 &= \omega\,\sin(\omega y)\big(\cos(\omega z)-\cos(\omega x)\big),\\
u^\dagger_3 &= \omega\,\sin(\omega z)\big(\cos(\omega x)-\cos(\omega y)\big).
\end{aligned}
\]

\paragraph{Evaluation metrics and empirical rates.}
We measure the $L^2(\Omega)$ error in both absolute and relative forms:
\[
    \mathrm{Err}_{L^2}^{\mathrm{abs}}
    = \|\mathbf{u}_n-\mathbf{u}^\dagger\|_{L^2(\Omega)},
    \qquad
    \mathrm{Err}_{L^2}^{\mathrm{rel}}
    = \frac{\|\mathbf{u}_n-\mathbf{u}^\dagger\|_{L^2(\Omega)}}{\|\mathbf{u}^\dagger\|_{L^2(\Omega)}}.
\]
Given an increasing sequence of neuron numbers $n_1<n_2<\cdots$, we estimate the empirical convergence rate between two successive sizes by the log--log slope
\[
\mathrm{rate}_i
=
-\frac{\log\!\bigl(\mathrm{Err}^{\mathrm{rel}}_i/\mathrm{Err}^{\mathrm{rel}}_{i-1}\bigr)}{\log\!\bigl(n_i/n_{i-1}\bigr)}.
\]
We report $\mathrm{Err}_{L^2}^{\mathrm{rel}}$ together with the corresponding empirical rates. By Theorem~\ref{thm:approximation-rates-of-div-free-FNS}, the proven upper bound predicts
\[
\mathrm{Err}_{L^2}^{\mathrm{rel}} = O\!\left(n^{-\frac12-\frac{2k-1}{2d}}\right).
\]
See the log--log plots in Figure~\ref{fig:l2_fits_d=23-comparek} and the empirical rate tables in Table~\ref{tab:l2_rates_d2_k1_k4} for $d=2$ and Table~\ref{tab:l2_rates_d3_k1_k4} for $d=3$.

\begin{figure}[htbp]
  \centering
  \includegraphics[width=0.49\linewidth]{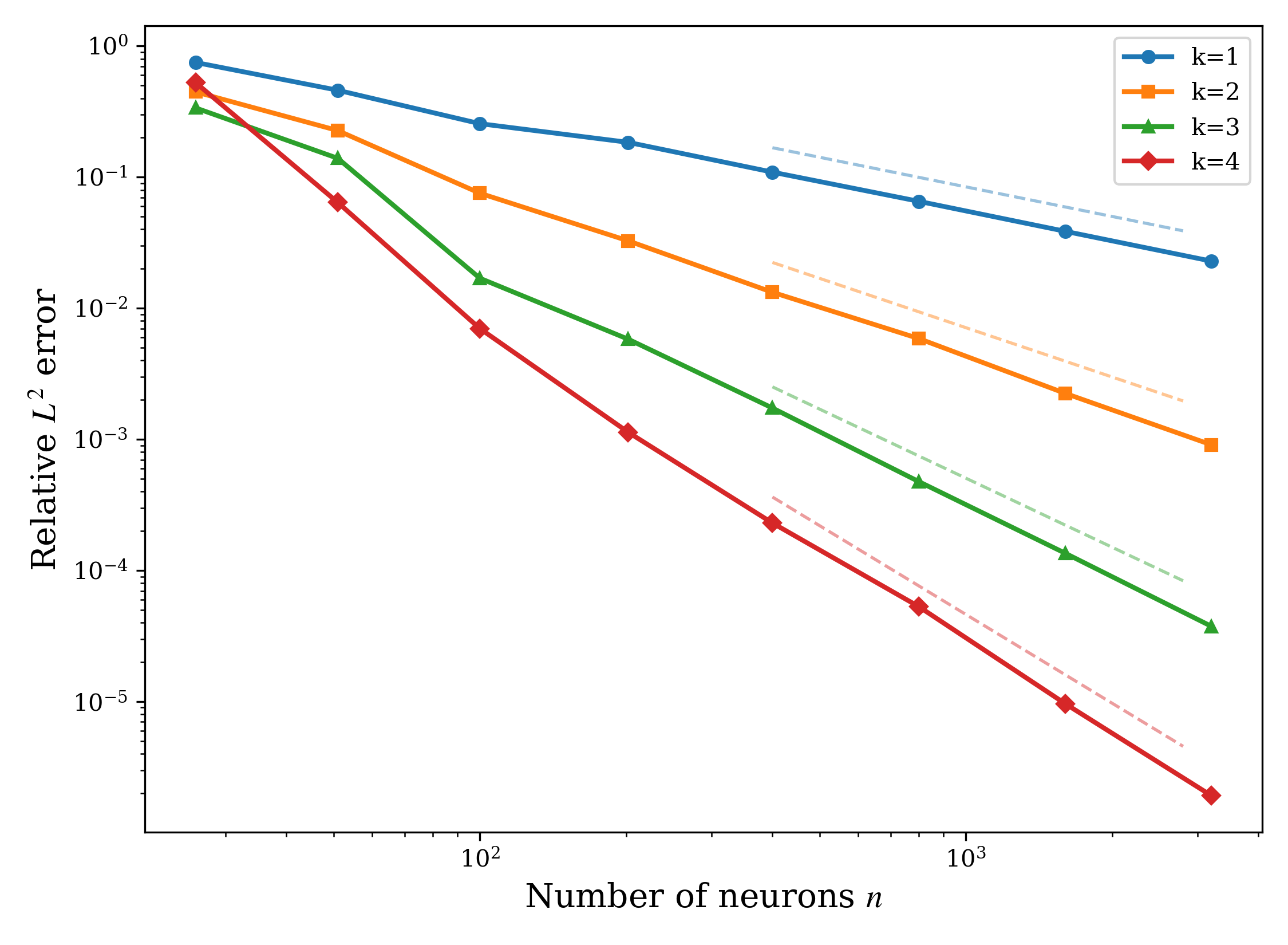}
  \includegraphics[width=0.49\linewidth]{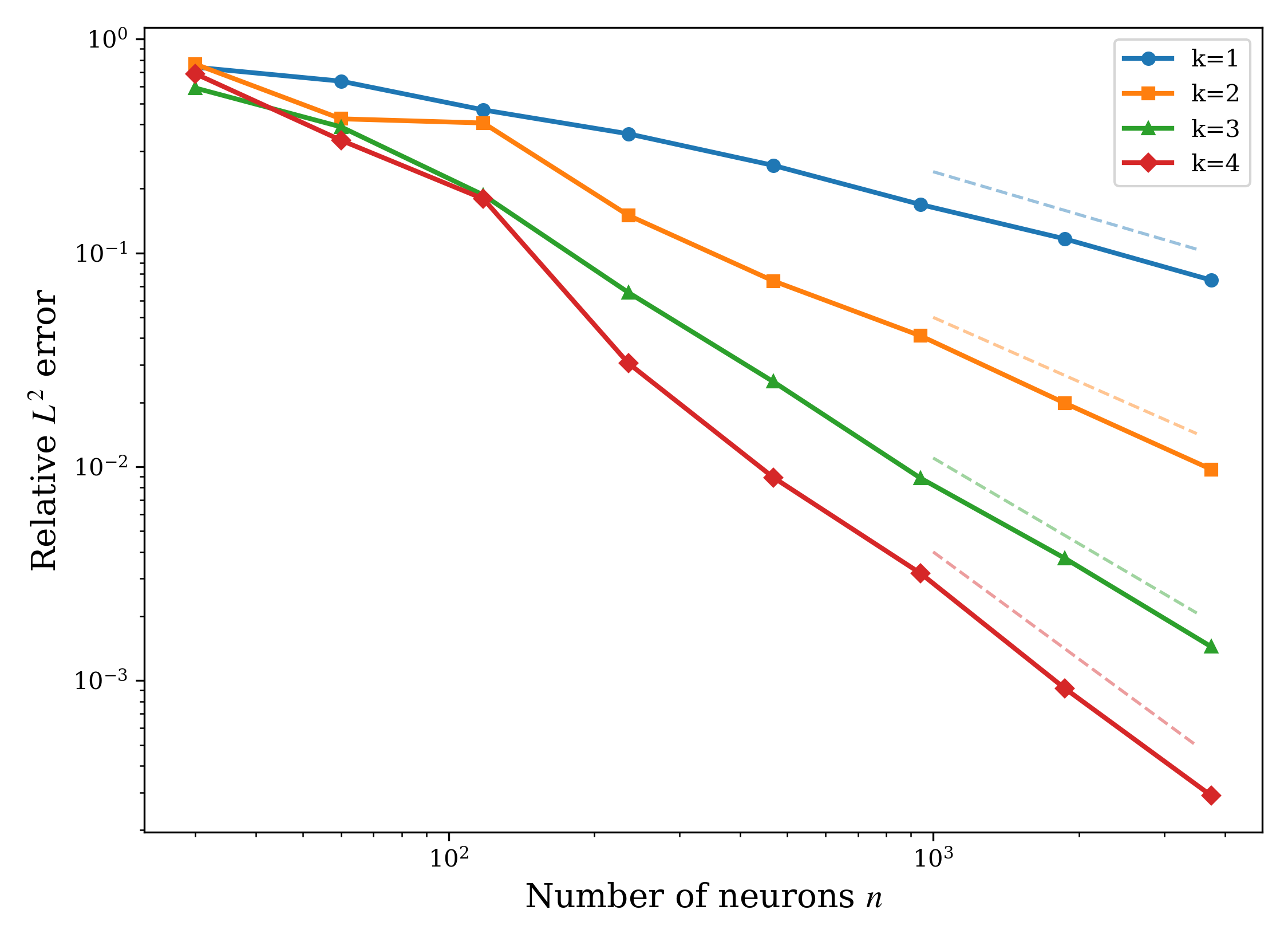}
  \caption{Divergence-free $L^2$ approximation: relative $L^2$ error decay with number of neurons.
  (Left) $d=2$. (Right) $d=3$. Each solid curve is accompanied by a dashed line of corresponding color, whose slope equals the theoretical upper bound on the convergence rate.}
  \label{fig:l2_fits_d=23-comparek}
\end{figure}

% -----------------------------------------
\subsection{Stokes equation with homogeneous boundary condition}\label{subsec:num-exp-homogeneous-stokes}
% -----------------------------------------\

\paragraph{Setup.}
We solve the steady Stokes equation with $\nu =1$ on $\Omega=[-1,1]^d$ ($d\in\{2,3\}$):
\[
-\nu\Delta \mathbf{u} + \nabla p = \mathbf{f},\qquad \nabla\cdot \mathbf{u}=0 \ \text{in }\Omega,\qquad
\mathbf{u}=\mathbf{0}\ \text{on }\partial\Omega.
\]
We discretize with the divergence-free finite neuron space $\VFNS$,
and test activation powers
\[
k\in\{2,3,4,5\}\qquad \text{(for both $d=2$ and $d=3$)}.
\]
We enforce the homogeneous boundary condition via a penalty parameter $\varepsilon>0$ and minimize the
penalized Stokes energy over $\VFNS$:
\[
\mathbf{u}_n=\arg\min_{\mathbf{v}\in V_n} J_\varepsilon(\mathbf{v})
=
\arg\min_{\mathbf{v}\in V_n}
\left\{
\frac{\nu}{2}\int_\Omega |\nabla\mathbf{v}|^2\,dx
-\int_\Omega \mathbf{f}\cdot\mathbf{v}\,dx
+\frac{1}{2\varepsilon}\int_{\partial\Omega}|\mathbf{v}|^2\,ds
\right\}.
\]
We also solve the weighted least-squares system directly in $d=2$ and using the normal equations in $d=3$. We impose the homogeneous Dirichlet condition via a boundary penalty
$\frac{1}{2\varepsilon}\int_{\partial\Omega}|\mathbf{u}_M|^2\,ds$ with parameter $\lambda_{\partial \Omega} = 1/\varepsilon$.
Unless otherwise stated, we set $\varepsilon=1$ for $d=2$ and $\varepsilon=1/6$ for $d=3$.
We do not tune $\varepsilon$. A sensitivity study in Section~\ref{subsubsec:boundary-penalty-sens}
shows that the convergence rates are insensitive to this choice.

\paragraph{Manufactured solutions and forcing.}
To generate the forcing term, we prescribe a divergence-free ground truth velocity $\mathbf{u}^\dagger$
satisfying $\mathbf{u}^\dagger|_{\partial\Omega}=\mathbf{0}$ and set
\[
\mathbf{f} := -\nu\Delta \mathbf{u}^\dagger.
\]
In implementation, we use the numerically preferred weak-form identity
\[
\int_\Omega \mathbf{f}\cdot\boldsymbol{\phi}\,dx
=
\nu\int_\Omega \nabla\mathbf{u}^\dagger : \nabla\boldsymbol{\phi}\,dx,
\quad \forall \boldsymbol{\phi}\in \VFNS,
\]
which avoids explicitly forming $\Delta \mathbf{u}^\dagger$.

To ensure zero Dirichlet boundary conditions, we use a bubble function
\[
\beta(t)=(1-t^2)^2,\qquad t\in[-1,1],
\]
and denote $\beta(x,y)=\beta(x)\beta(y)$, $\beta(x,y,z)=\beta(x)\beta(y)\beta(z)$.
We use a frequency parameter $\omega$ (default $\omega=\pi$).

\textbf{2D Stokes target.}
Let
\[
\psi(x,y)=\beta(x)\beta(y)\,\sin(\omega x)\sin(\omega y),
\qquad
\mathbf{u}^\dagger_{2D}=\nabla^\perp\psi
=
\begin{pmatrix}
\partial_y\psi\\
-\partial_x\psi
\end{pmatrix}.
\]
By construction $\nabla\cdot \mathbf{u}^\dagger_{2D}=0$ and $\mathbf{u}^\dagger_{2D}|_{\partial\Omega}=0$.

\textbf{3D Stokes target.}
Define the antisymmetric potential matrix $\mu(x)\in\mathbb{R}^{3\times 3}$ with entries
\[
\mu(x,y,z)
=
\beta(x,y,z)
\begin{pmatrix}
0 & \sin(\omega x)\sin(\omega y) & -\sin(\omega x)\sin(\omega z)\\
-\sin(\omega x)\sin(\omega y) & 0 & \sin(\omega y)\sin(\omega z)\\
\sin(\omega x)\sin(\omega z) & -\sin(\omega y)\sin(\omega z) & 0
\end{pmatrix}.
\]
We then define the divergence-free velocity field by the antisymmetric-divergence operator
\[
(\mathbf{u}^\dagger_{3D})_i = \sum_{j=1}^3 \partial_j \mu_{ij}.
\]
Equivalently, writing
\[
T_x = \partial_x\!\big(\beta(x,y,z)\sin(\omega x)\big),\quad
T_y = \partial_y\!\big(\beta(x,y,z)\sin(\omega y)\big),\quad
T_z = \partial_z\!\big(\beta(x,y,z)\sin(\omega z)\big),
\]
we obtain the explicit symmetric form
\[
\mathbf{u}^\dagger_{3D}
=
\begin{pmatrix}
\sin(\omega x)\,(T_y-T_z)\\
\sin(\omega y)\,(T_z-T_x)\\
\sin(\omega z)\,(T_x-T_y)
\end{pmatrix}.
\]
This construction ensures $\nabla\cdot \mathbf{u}^\dagger_{3D}=0$ and $\mathbf{u}^\dagger_{3D}|_{\partial\Omega}=0$.

\paragraph{Evaluation metrics and empirical rates.}
We report the relative errors in the $L^2(\Omega)$ norm and the $\dot H^1(\Omega)$ seminorm:
\[
\mathrm{Err}_{\dot{H}^1}^{\mathrm{rel}}
    = \frac{\|\nabla(\mathbf{u}_n-\mathbf{u}^\dagger)\|_{L^2(\Omega)}}{\|\nabla\mathbf{u}^\dagger\|_{L^2(\Omega)}},
\qquad
\mathrm{Err}_{L^2}^{\mathrm{rel}}
    = \frac{\|\mathbf{u}_n-\mathbf{u}^\dagger\|_{L^2(\Omega)}}{\|\mathbf{u}^\dagger\|_{L^2(\Omega)}}.
\]
When sweeping $n$, we estimate empirical convergence rates for each metric by the corresponding log--log slope.
By Theorem~\ref{thm:approximation-rates-of-div-free-FNS}, the proven upper bounds on the relative errors scale as
\[
\mathrm{Err}_{L^2}^{\mathrm{rel}} = O\!\left(n^{-\frac12-\frac{2k-1}{2d}}\right),
\qquad
\mathrm{Err}_{\dot H^1}^{\mathrm{rel}} = O\!\left(n^{-\frac12-\frac{2k-3}{2d}}\right).
\]
See the log--log plots in Figure~\ref{fig:stokes_solve_d=23-comparek} and the empirical rate tables in~\Cref{tab:stokes_h1_rates_d2_k2_k5,tab:stokes_l2_rates_d2_k2_k5} for $d=2$ and ~\Cref{tab:stokes_h1_rates_d3_k2_k5,tab:stokes_l2_rates_d3_k2_k5} for $d=3$.

\begin{figure}[htbp]
  \centering
  \includegraphics[width=0.49\linewidth]{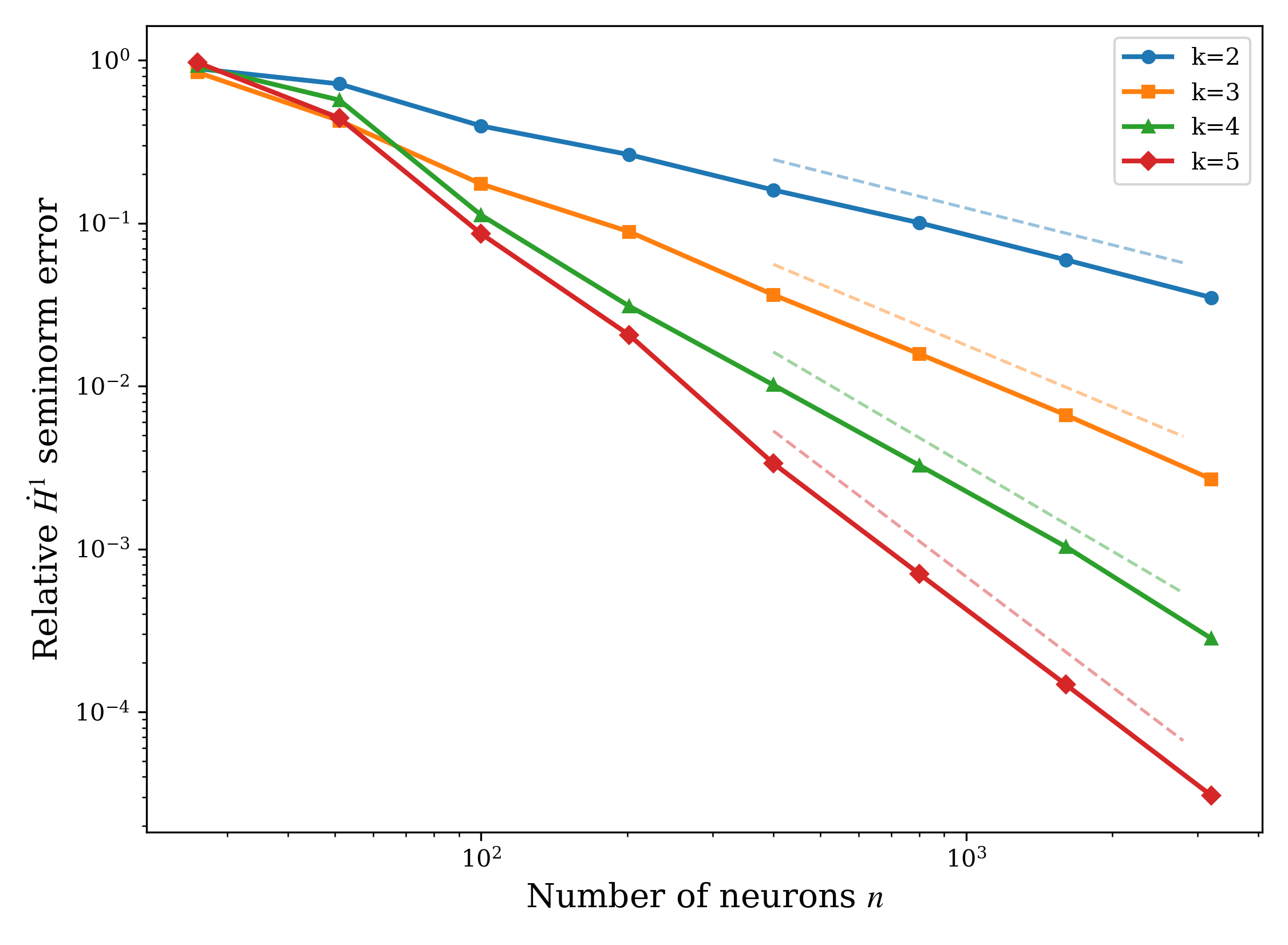}
  \includegraphics[width=0.49\linewidth]{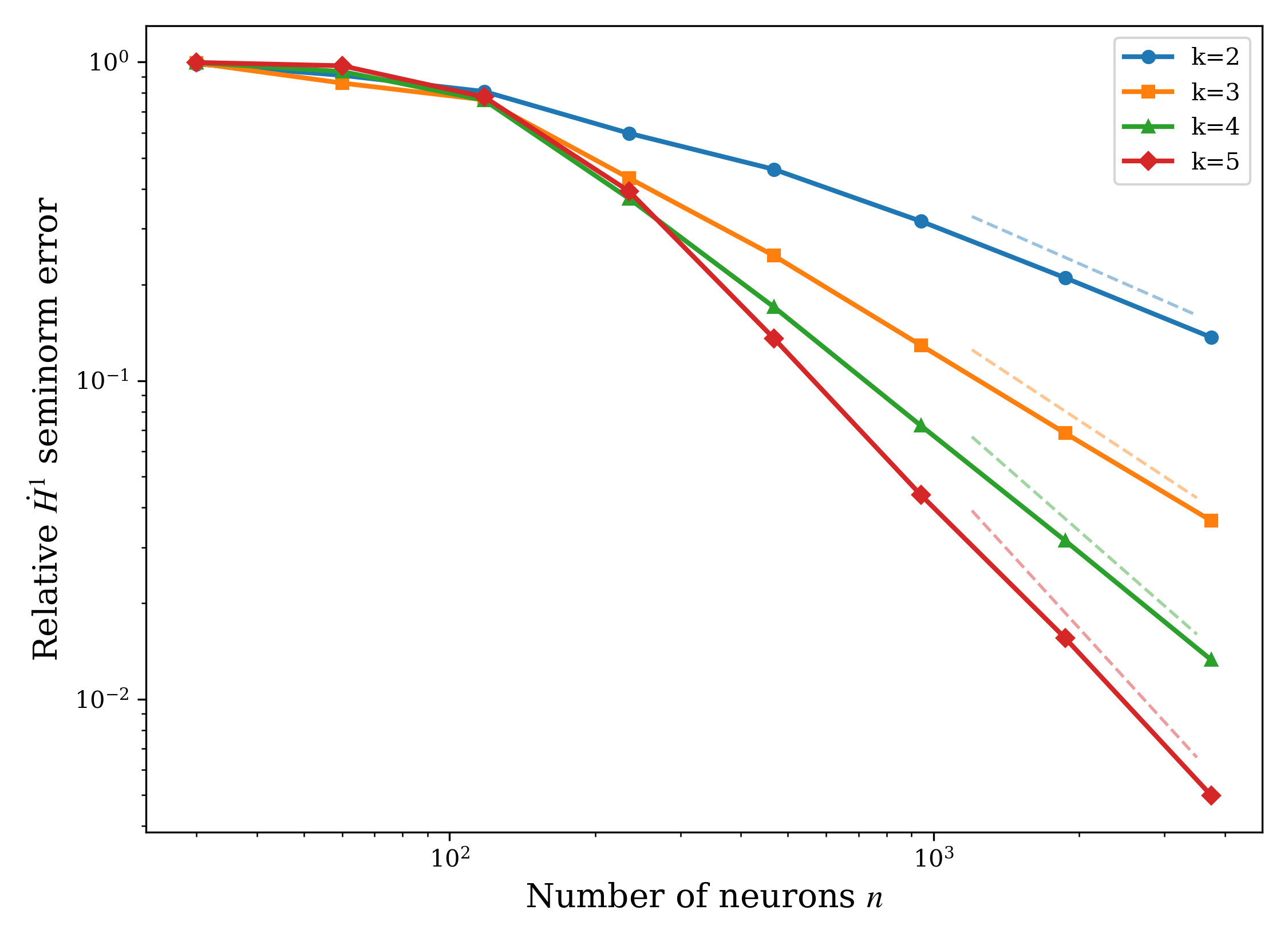}
  \caption{Solving Stokes equation in div-free FNS: relative $\dot{H}^1$ seminorm error decay with number of neurons.
  (Left) $d=2$. (Right) $d=3$. Dashed lines in matching colors show the theoretical upper bound slopes for reference.}
  \label{fig:stokes_solve_d=23-comparek}
\end{figure}

\subsection{Lid-driven cavity flow: a non-manufactured Stokes benchmark}

To demonstrate that the divergence-free FNS discretization is not limited to manufactured solutions with known closed-form targets, we consider the classical 2D lid-driven cavity flow for the steady Stokes system on the square domain $\Omega=[-1,1]^2$. This benchmark is driven entirely by non-homogeneous boundary data and therefore provides a more realistic test of whether the method remains effective beyond synthetic examples.

We solve
\[
-\nu \Delta u + \nabla p = 0 \quad \text{in }\Omega,
\qquad
\nabla\cdot u = 0 \quad \text{in }\Omega,
\]
with no-slip boundary conditions on the cavity walls except for a prescribed horizontal lid velocity on the top boundary. More precisely,
\[
u = 0 \quad \text{on } \partial\Omega\setminus\Gamma_{\mathrm{top}},
\qquad
u(x,1) = (g(x),\,0) \quad \text{on }\Gamma_{\mathrm{top}}:=\{(x,1):x\in[-1,1]\}.
\]
We test two representative lid profiles:
\[
g_{\mathrm{const}}(x)=1,
\qquad
g_{\mathrm{smooth}}(x)=\sin^2\!\Bigl(\frac{\pi}{2}(x+1)\Bigr).
\]
The first is the classical lid-driven cavity boundary condition; however, it is discontinuous at the top corners when matched with the no-slip side walls, and is therefore expected to induce reduced regularity near the boundary. The second vanishes smoothly at the two top corners and is used as a regularized lid profile to alleviate this singular behavior.

\paragraph{Divergence-free FNS discretization.}
We use the same divergence-free trial space $V_n^k\subset H^1_{\mathrm{div}}(\Omega)$ as in the previous Stokes experiments and impose the non-homogeneous Dirichlet data weakly through a boundary penalty. Let $u_{\mathrm{bc}}$ denote the prescribed boundary velocity, namely zero on $\partial\Omega\setminus\Gamma_{\mathrm{top}}$ and $(g(x),0)$ on $\Gamma_{\mathrm{top}}$. We then minimize over $V_n^k$
\[
u_n = \arg\min_{v\in V_n^k}
\left\{
\frac{\nu}{2}\int_{\Omega}|\nabla v|^2\,dx
+\frac{1}{2\varepsilon}\int_{\partial\Omega}|v-u_{\mathrm{bc}}|^2\,ds
\right\},
\]
which yields a convex quadratic problem in the outer coefficients. As in the previous experiments, the incompressibility constraint is built into the trial space and is therefore satisfied exactly at the discrete level.

\paragraph{FEM reference solution.}
As no analytic solution is available for this benchmark, we compute a high-resolution finite element reference solution using \texttt{dolfinx}. On $\Omega=[-1,1]^2$, we generate a uniform triangular mesh of size $512\times512$ and use the Taylor--Hood pair $(P_2,P_1)$, namely continuous piecewise quadratic velocity and continuous piecewise linear pressure. The steady Stokes system is solved with zero body force using the standard mixed weak formulation
\[
a((u,p),(v,q))
=
\nu(\nabla u,\nabla v) - (p,\nabla\cdot v) - (\nabla\cdot u,q),
\qquad
L(v)=0,
\]
together with Dirichlet boundary conditions given by the lid velocity on the top edge and no-slip on the remaining walls. The resulting saddle-point linear system is solved by a direct LU factorization through PETSc with MUMPS.

\paragraph{Evaluation.}
We report errors in the $L^2(\Omega)$ norm and the $\dot H^1(\Omega)$ seminorm with respect to the FEM reference solution $\mathbf{u}^\dagger$, together with the empirical convergence rates as the number of neurons increases.
\[
\mathrm{Err}_{L^2}
    = \|\mathbf{u}_n-\mathbf{u}^\dagger\|_{L^2(\Omega)},
    \qquad
\mathrm{Err}_{\dot{H}^1}
    = \|\nabla(\mathbf{u}_n-\mathbf{u}^\dagger)\|_{L^2(\Omega)}.
\]
For the classical constant-lid case, we additionally evaluate the error on a trimmed interior subdomain $\Omega'=[-0.9,0.9]^2$ to get $\mathrm{Err}_{L^2}^{\mathrm{inner}}$ and $\mathrm{Err}_{\dot{H}^1}^{\mathrm{inner}}$, in order to separate the global approximation quality from the loss of regularity near the boundary and, in particular, near the two top corners.

\subsubsection{Classical lid-driven cavity flow}
% We first consider the standard choice $g_{\mathrm{const}}(x)=1$. In this case, the horizontal lid velocity is incompatible with the no-slip side walls at the two top corners, so the solution is expected to exhibit reduced regularity there. This behavior is clearly reflected in the numerical results. The predicted flow pattern is qualitatively correct: the streamline and velocity plots recover the main cavity structure and are visually close to the FEM reference. However, the global convergence rates are noticeably degraded, especially in the $H^1$ norm. 
We first consider the standard choice $g_{\mathrm{const}}(x)=1$. In this case, the horizontal lid velocity is incompatible with the no-slip side walls at the two top corners, so reduced regularity near the boundary is expected. From the viewpoint of this work, this example is a diagnostic result of independent interest. It shows that the proposed divergence-free FNS discretization is applicable to a standard non-manufactured benchmark, while also revealing its sensitivity to the regularity of the underlying boundary-driven Stokes solution, in a manner consistent with theoretical intuition.
\begin{figure}[htbp]
  \centering
  \includegraphics[width=0.9\textwidth]{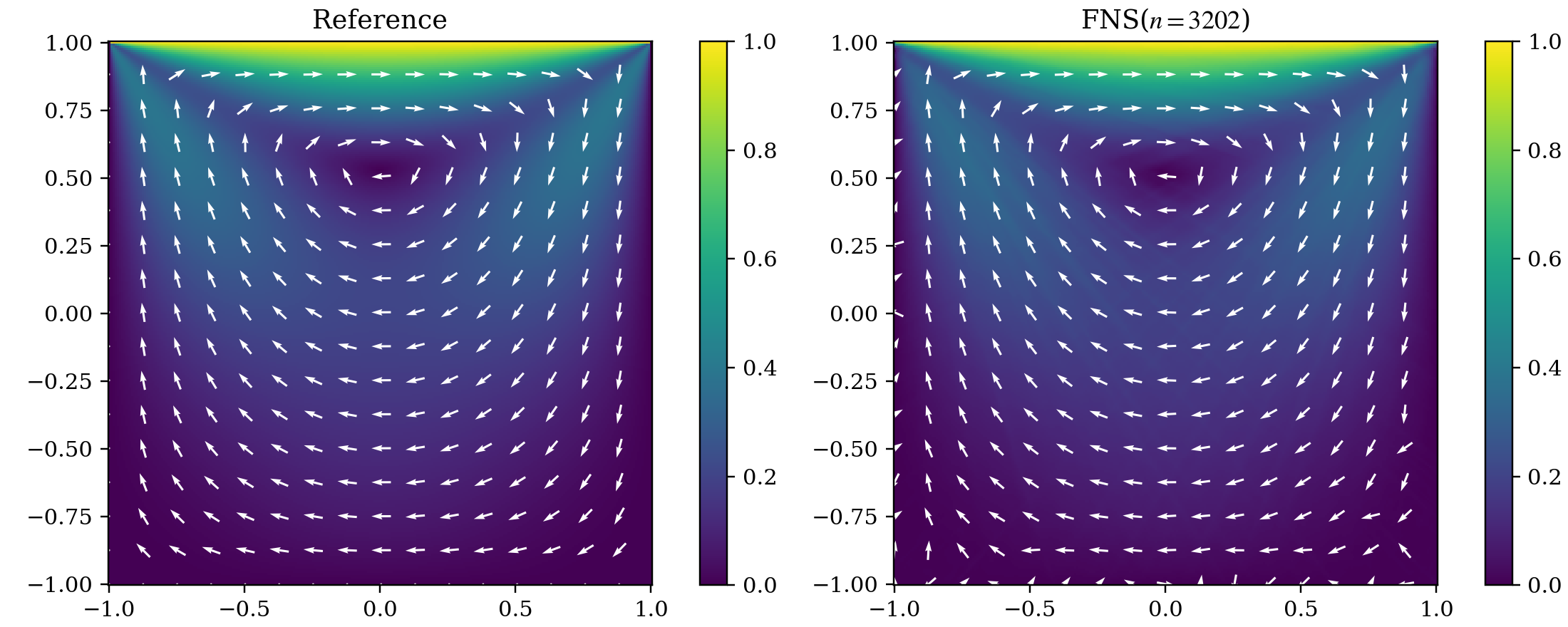}
  \caption{Solving classical lid-driven cavity flow by div-free FNS, $k=2, n=3202$.}
  \label{fig:classic-lid-flow}
\end{figure}

This behavior is clearly reflected in the numerical results. The predicted flow field remains qualitatively correct: the main cavity structure is captured and the streamline and velocity plots are visually close to the FEM reference. However, the global convergence rates are substantially degraded, especially in the $\dot H^1$ norm. See Table~\ref{tab:classic_lid_k2}, for $k=2$, the relative $L^2$ error still decreases at an empirical rate around $0.6$, whereas the global $\dot H^1$ rate remains below $0.1$. When the error is evaluated instead on an interior subdomain $\Omega'=[-0.9,0.9]^2$, the observed $\dot H^1$ rate improves to about $0.3$. This marked improvement after interior trimming indicates that the dominant difficulty is localized near the boundary, and in particular near the two top corners, rather than in the bulk flow.
\begin{table}[htbp]
\centering
\caption{Solving classical lid-driven cavity flow by divergence-free FNS: errors and empirical rates for $k=2$.}
\label{tab:classic_lid_k2}
\small
\setlength{\tabcolsep}{3pt}

\begin{minipage}[t]{0.48\textwidth}
\centering
\textbf{(a) $L^2$ errors}

\vspace{0.3em}
\begin{tabular}{c cc cc}
\hline
& \multicolumn{2}{c}{global}
& \multicolumn{2}{c}{inner} \\
\cline{2-3}\cline{4-5}
$n$
& $\mathrm{Err}_{L^2}$ & --
& $\mathrm{Err}_{L^2}^{\mathrm{inner}}$ & -- \\
\hline
51   & 2.840e-01 & --    & 2.654e-01 & --    \\
100  & 2.153e-01 & 0.41 & 1.935e-01 & 0.47 \\
202  & 1.384e-01 & 0.62 & 1.227e-01 & 0.65 \\
400  & 7.691e-02 & 0.85 & 6.158e-02 & 1.01 \\
801  & 7.344e-02 & 0.07 & 5.909e-02 & 0.06 \\
1604 & 5.077e-02 & 0.53 & 3.991e-02 & 0.57 \\
3202 & 4.084e-02 & 0.31 & 3.122e-02 & 0.35 \\
\hline
\end{tabular}
\end{minipage}
\hfill
\begin{minipage}[t]{0.48\textwidth}
\centering
\textbf{(b) $\dot H^1$ errors}

\vspace{0.3em}
\begin{tabular}{c cc cc}
\hline
& \multicolumn{2}{c}{global}
& \multicolumn{2}{c}{inner} \\
\cline{2-3}\cline{4-5}
$n$
& $\mathrm{Err}_{\dot H^1}$ & --
& $\mathrm{Err}_{\dot H^1}^{\mathrm{inner}}$ & -- \\
\hline
51   & 5.343e+00 & --     & 2.256e+00 & --     \\
100  & 5.533e+00 & -0.05 & 2.626e+00 & -0.23 \\
202  & 5.114e+00 & 0.11  & 2.072e+00 & 0.34  \\
400  & 4.698e+00 & 0.12  & 1.681e+00 & 0.31  \\
801  & 4.684e+00 & 0.00  & 1.581e+00 & 0.09  \\
1604 & 4.351e+00 & 0.11  & 1.290e+00 & 0.29  \\
3202 & 4.032e+00 & 0.11  & 9.483e-01 & 0.44  \\
\hline
\end{tabular}
\end{minipage}

\end{table}

\subsubsection{Regularized lid-driven cavity flow}\label{subsubsec:regularized-lid}
We next replace the lid with the regularized profile
\[
g_{\mathrm{smooth}}(x)=\sin^2\!\Bigl(\frac{\pi}{2}(x+1)\Bigr),
\]
which vanishes smoothly at the two top corners and is therefore more compatible with the no-slip boundary condition on the side walls. In this case, the convergence behavior improves substantially, see \Cref{tab:reg_lid_k2_k3}. 
These results are consistent with the interpretation that the poor behavior in the classical cavity is not caused by the divergence-free FNS discretization itself, but rather by the low regularity of the induced solution. Once the lid profile is smoothed near the corners, the method recovers clear high-order convergence trends.
\begin{figure}[htbp]
  \centering
  \includegraphics[width=0.9\textwidth]{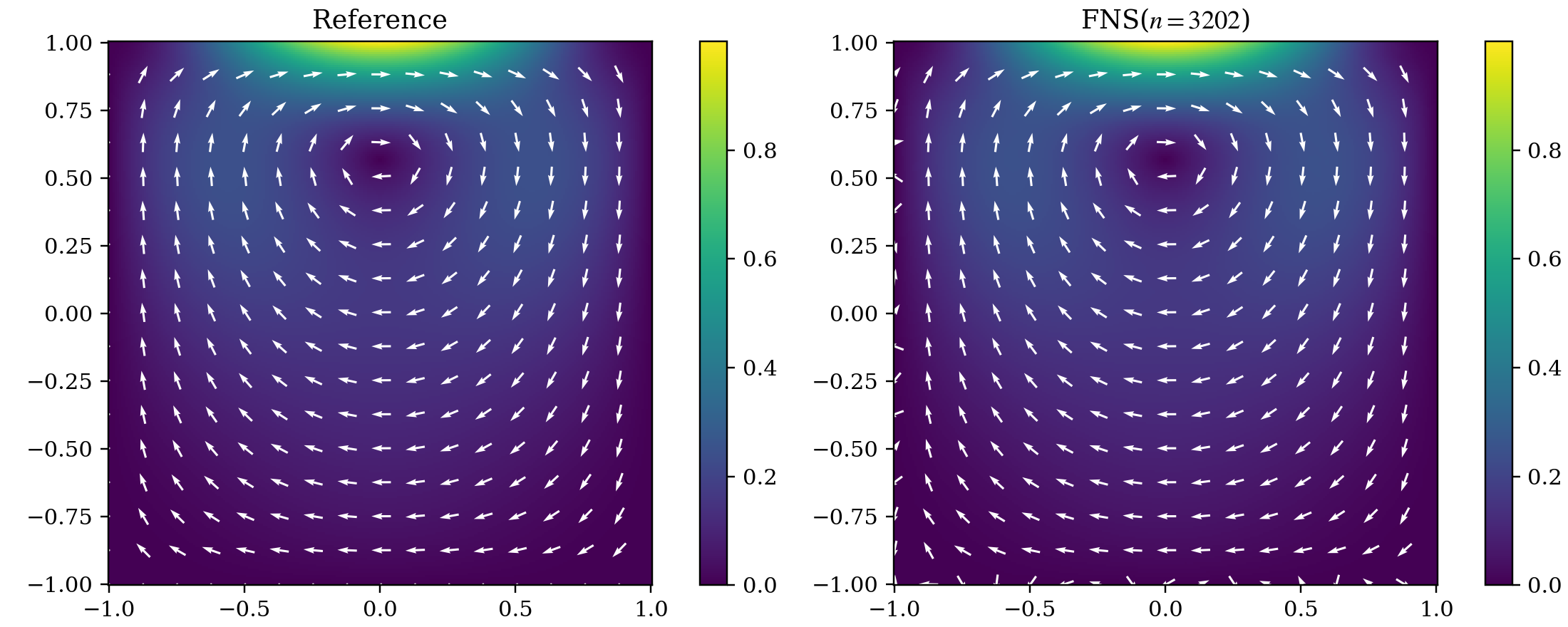}
  \caption{Solving regularized lid-driven cavity flow by divergence-free FNS, $k=2, n=3202$.}
  \label{fig:regularized-lid-flow}
\end{figure}

We compare the Taylor--Hood $P_2/P_1$ finite element baseline and the divergence-free FNS approximation at matched degrees of freedom (DOF). See Figure~\ref{fig:classic-lid-fem&fns-compare} for the classic lid case and Figure~\ref{fig:regularized-lid-fem&fns-compare} for the regularized lid case. Here, DOF refers to the dimension of the resulting linear system, namely the number of unknown linear coefficients in the discretization. For the smooth cavity benchmark considered here, and within the tested DOF range, the divergence-free FNS approximation yields slightly smaller errors than the Taylor--Hood $P_2/P_1$ baseline.

\begin{figure}[htbp]
  \centering
  \includegraphics[width=0.49\textwidth]{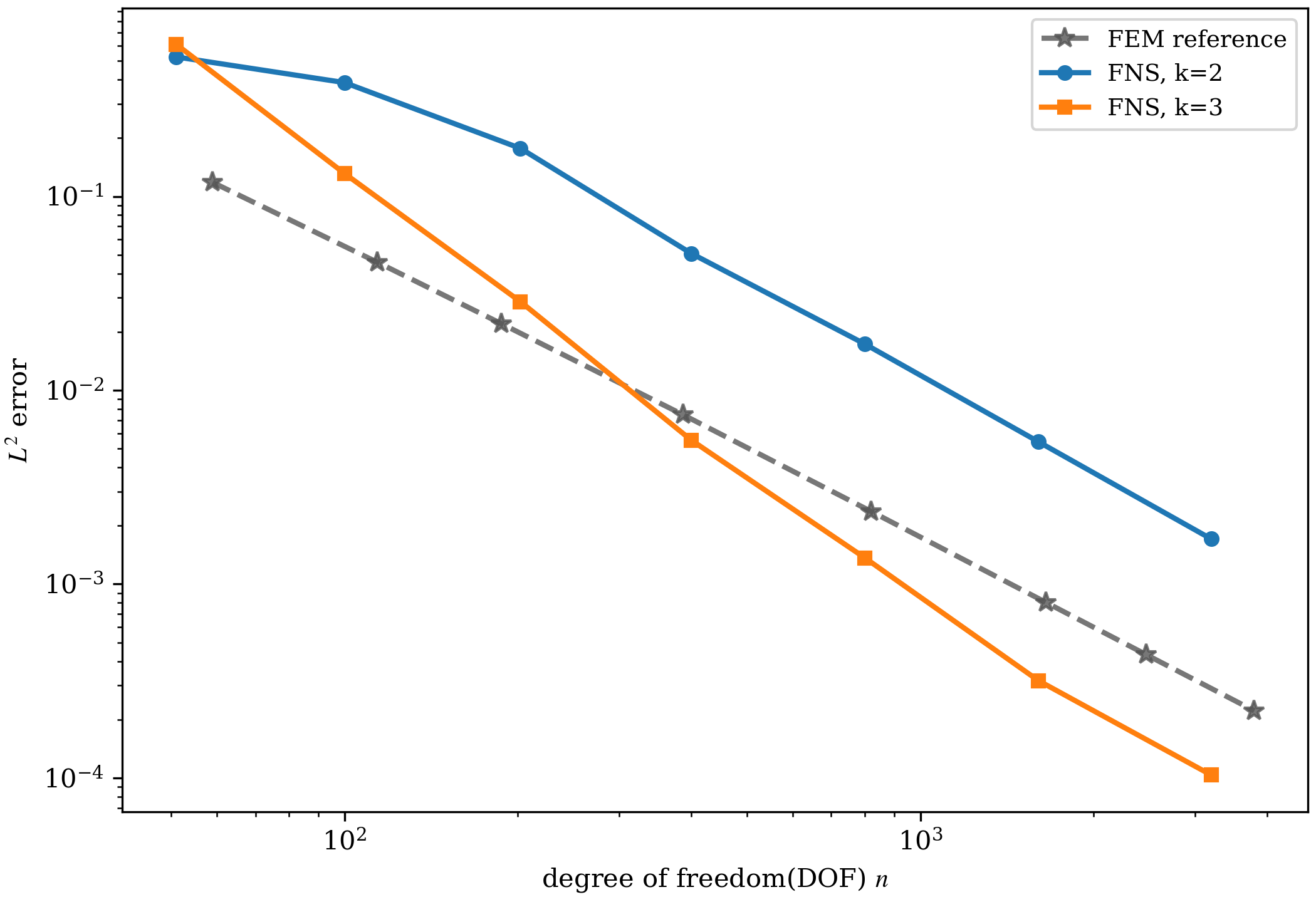}
  \includegraphics[width=0.49\textwidth]{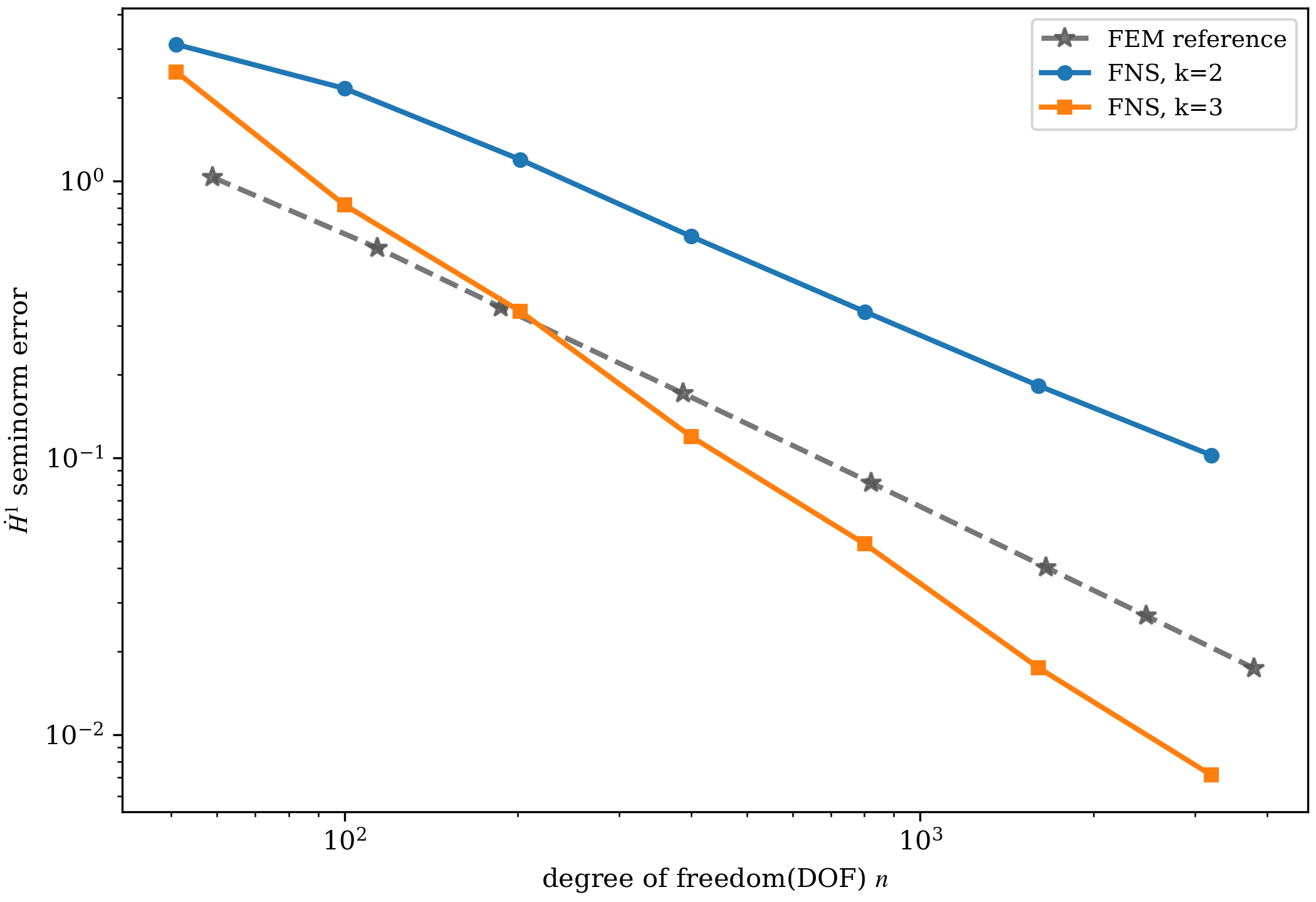}
  \caption{Solving regularized lid-driven cavity flow by divergence-free FNS compared with finite element method: $L^2$ error and $\dot{H}^1$ seminorm error decay with degree of freedom.
  Left: $L^2$ error. Right: $\dot{H}^1$ seminorm error.}
  \label{fig:regularized-lid-fem&fns-compare}
\end{figure}

Overall, the lid-driven cavity experiments complement the manufactured-solution tests from the previous subsection. They show that the proposed method remains applicable on a standard non-manufactured Stokes benchmark, while also making clear that its practical convergence behavior is sensitive to the regularity of the boundary data.

\subsection{Additional Experiments}

\subsubsection{Qualitative visualization of 2D cases}
\label{subsec:vis-2d-flow}

To complement the quantitative $L^2$ errors and empirical rates, we visualize the learned divergence-free velocity fields on the square domain $\Omega=[-1,1]^2$.
Given a trained model with $n$ neurons producing a vector field $\mathbf{u}_n:\Omega\to\mathbb{R}^2$, we evaluate it on a uniform grid and display both its direction and magnitude.
Specifically, we render a background heatmap of the speed $|\mathbf{u}_n(x)|$ and overlay a dense quiver plot to indicate the flow direction. The arrow length is normalized so that the heatmap solely encodes the magnitude.
All panels sharing the same physical quantity use a shared color scale to enable direct visual comparison across different numbers of neurons $n$.

The figure is arranged to compare (i) the ground-truth field $\mathbf{u}^\dagger$ and (ii) three learned approximations with increasing model size, illustrating progressively improved agreement in both global structure and local details.
To localize the remaining discrepancy, we additionally show the pointwise error field
\[
    e(x) := \|\mathbf{u}_n(x)-\mathbf{u}^\dagger(x)\|_2,
\]
as a heatmap with its own color scale.
We show the divergence-free $L^2$ approximation when $d=2, k=2$ in Figure.~\ref{fig:divfree2d-flow} and stokes equation when $d=2, k=2$ in Figure.~\ref{fig:stokes2d-flow}.

\subsubsection{Comparison between Normal equations and  direct least squares}\label{subsubsec:mass_ablation}
We compare the normal-equation solver and the direct least-squares solver on the 2D divergence-free $L^2$ approximation with $k=4$. The $2$-norm condition number $\kappa_2$ is computed via singular value decomposition ($\texttt{scipy.linalg.svd}$) rather than derived from the least-squares solver ($\texttt{scipy.linalg.lstsq}$), as the latter tends to truncate small singular values and may underestimate the ill-conditioning. 

As shown in Figure~\ref{fig:mass-ablation} (right), the condition number of the normal equation system grows rapidly,
consistent with the squaring effect $\kappa_2(H^\top W^\top W H)\approx \kappa_2(WH)^2$.
Correspondingly, Figure~\ref{fig:mass-ablation} (left) shows that the normal-equation approach exhibits error saturation at large $n$, while the direct LS solver continues to decrease the loss and remains consistent with the predicted rate in the tested regime.
This indicates that, beyond a moderate width, the observed accuracy of the mass-matrix approach is limited by numerical stability rather than approximation power.
\begin{figure}
    \centering
    \includegraphics[width=0.48\linewidth]{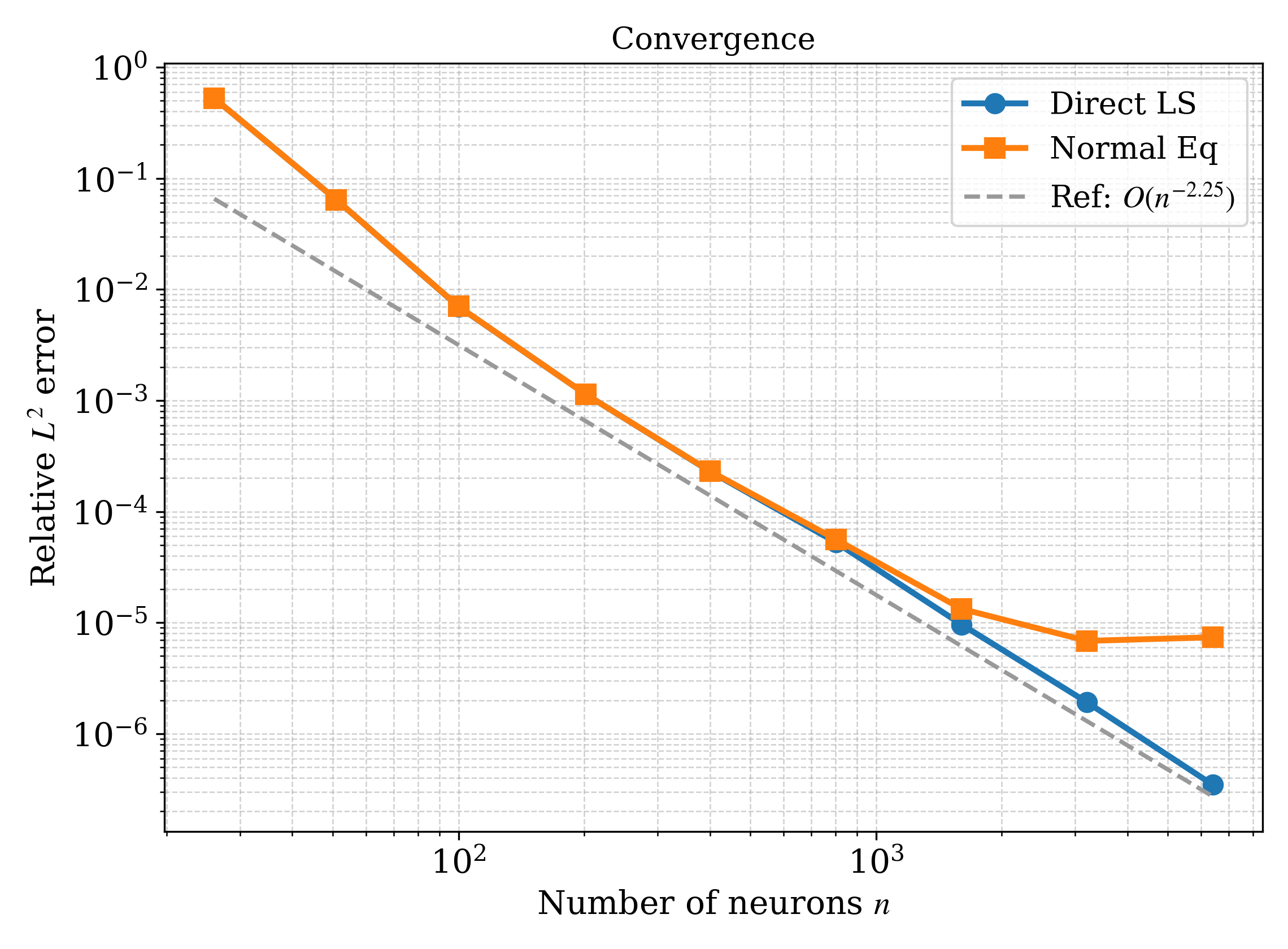}
    \includegraphics[width=0.48\linewidth]{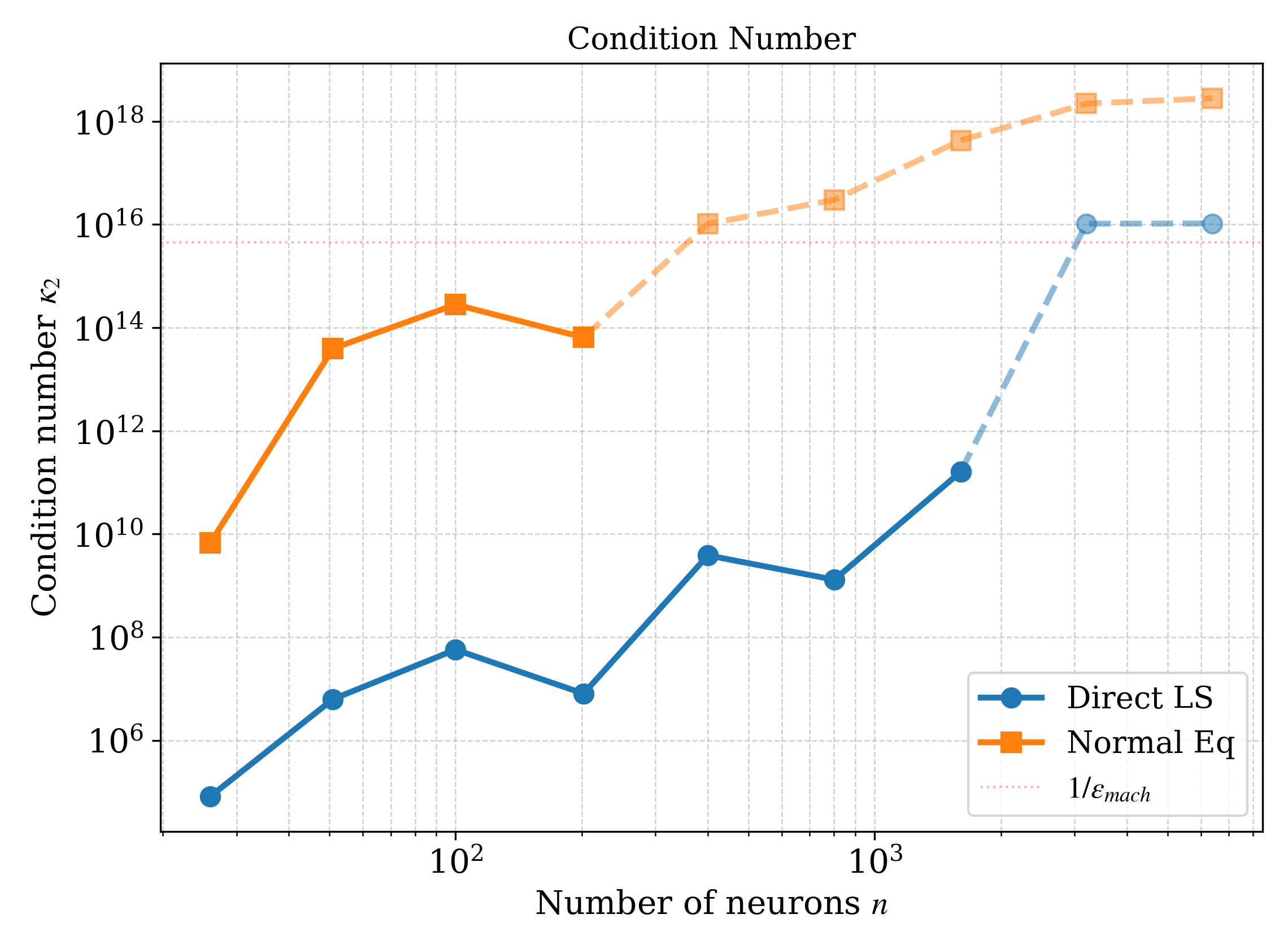}
    \caption{$d=2$, $k=4$ divergence-free $L^2$ approximation: error decay (left) and condition number (right), comparing normal equations (mass matrix) vs.\ direct least squares.}
    \label{fig:mass-ablation}
\end{figure}

\subsubsection{Sensitivity to the boundary penalty parameter}\label{subsubsec:boundary-penalty-sens}
As discussed in Section~\ref{subsec:num-exp-homogeneous-stokes}, we enforce the homogeneous Dirichlet condition weakly via the penalty term
$\frac{1}{2\varepsilon}\int_{\partial\Omega}|\mathbf{u}_M|^2\,ds$.
In all main experiments we fix $\varepsilon=1$ for $d=2$ and $\varepsilon=1/6$ for $d=3$.

To assess sensitivity, we rerun the Stokes experiments at $k=4$ with multiple $\varepsilon$ values.
For Stokes2D we test $\varepsilon\in\{0.01,\,100\}$ (in addition to the default $\varepsilon=1$),
and for Stokes3D we test $\varepsilon\in\{0.06,\,600\}$ (in addition to the default $\varepsilon=1/6$).
\Cref{tab:stokes2d_lambda_sensitivity,tab:stokes3d_lambda_sensitivity} report the boundary residual
$\|\mathbf{u}_M\|_{L^2(\partial\Omega)}$ as well as the empirical $H^1$- and $L^2$-rates.

Across all tested $\varepsilon$, the $\dot H^1$ and $L^2$ convergence rates are essentially unchanged,
while $\|\mathbf{u}_n\|_{L^2(\partial\Omega)}$ decreases mildly as $\varepsilon$ becomes smaller
(i.e., a stronger penalty). This indicates that our results are robust to the choice of $\varepsilon$,
hence we keep the above fixed defaults.

\section{Conclusion}\label{sec:conclusion}
The construction of function spaces that inherently satisfy physical constraints offers significant advantages. By viewing the Finite Neuron Space (FNS) as a Quasi-Monte Carlo approximation of structure-preserving kernels, we developed a framework for divergence-free shallow linear networks. This synthesis combines the theoretical rigor of vector-valued RKHS analysis, the constructive methods of differential operators, and the computational efficiency of linearized neural networks. The resulting framework provides an exactly divergence-free, linearly trainable ansatz class whose approximation rates match those of the underlying scalar linearized-network theory under the stated domain and sampling assumptions.

\section{Extra Tables and Figures}\label{sec:extra-tables-figures}

\begin{table}[htbp]
\centering
\caption{Divergence-free $L^2$ approximation in $d=2$: relative $L^2$ errors and empirical rates.}
\label{tab:l2_rates_d2_k1_k4}
\small
\setlength{\tabcolsep}{4pt}
\begin{tabular}{c cc cc cc cc}
\hline
& \multicolumn{2}{c}{$k=1$}
& \multicolumn{2}{c}{$k=2$}
& \multicolumn{2}{c}{$k=3$}
& \multicolumn{2}{c}{$k=4$} \\
\cline{2-3}\cline{4-5}\cline{6-7}\cline{8-9}
$n$
& $\mathrm{Err}_{L^2}^{\mathrm{rel}}$ & $0.75$
& $\mathrm{Err}_{L^2}^{\mathrm{rel}}$ & $1.25$
& $\mathrm{Err}_{L^2}^{\mathrm{rel}}$ & $1.75$
& $\mathrm{Err}_{L^2}^{\mathrm{rel}}$ & $2.25$ \\
\hline
26   & 7.500e-01 & --    & 4.472e-01 & --    & 3.376e-01 & --    & 5.266e-01 & --    \\
51   & 4.603e-01 & 0.72 & 2.267e-01 & 1.01 & 1.393e-01 & 1.31 & 6.416e-02 & 3.12 \\
100  & 2.559e-01 & 0.87 & 7.544e-02 & 1.63 & 1.695e-02 & 3.13 & 7.005e-03 & 3.29 \\
202  & 1.843e-01 & 0.47 & 3.267e-02 & 1.19 & 5.807e-03 & 1.52 & 1.133e-03 & 2.59 \\
400  & 1.092e-01 & 0.77 & 1.325e-02 & 1.32 & 1.734e-03 & 1.77 & 2.302e-04 & 2.33 \\
801  & 6.539e-02 & 0.74 & 5.867e-03 & 1.17 & 4.756e-04 & 1.86 & 5.292e-05 & 2.12 \\
1604 & 3.859e-02 & 0.76 & 2.239e-03 & 1.39 & 1.350e-04 & 1.81 & 9.606e-06 & 2.46 \\
3202 & 2.288e-02 & 0.76 & 9.089e-04 & 1.30 & 3.750e-05 & 1.85 & 1.917e-06 & 2.33 \\
\hline
\end{tabular}
\end{table}

%%%%%%%%%%%%%

\begin{table}[htbp]
\centering
\caption{Divergence-free $L^2$ approximation in $d=3$: relative $L^2$ errors and empirical rates.}
\label{tab:l2_rates_d3_k1_k4}
\small
\setlength{\tabcolsep}{4pt}
\begin{tabular}{c cc cc cc cc}
\hline
& \multicolumn{2}{c}{$k=1$}
& \multicolumn{2}{c}{$k=2$}
& \multicolumn{2}{c}{$k=3$}
& \multicolumn{2}{c}{$k=4$} \\
\cline{2-3}\cline{4-5}\cline{6-7}\cline{8-9}
$n$
& $\mathrm{Err}_{L^2}^{\mathrm{rel}}$ & $0.67$
& $\mathrm{Err}_{L^2}^{\mathrm{rel}}$ & $1.00$
& $\mathrm{Err}_{L^2}^{\mathrm{rel}}$ & $1.33$
& $\mathrm{Err}_{L^2}^{\mathrm{rel}}$ & $1.66$ \\
\hline
30   & 7.397e-01 & --    & 7.649e-01 & --    & 5.914e-01 & --    & 6.899e-01 & --    \\
60   & 6.363e-01 & 0.22 & 4.241e-01 & 0.85 & 3.888e-01 & 0.60 & 3.366e-01 & 1.04 \\
118  & 4.668e-01 & 0.46 & 4.055e-01 & 0.07 & 1.867e-01 & 1.08 & 1.791e-01 & 0.93 \\
235  & 3.609e-01 & 0.37 & 1.500e-01 & 1.44 & 6.536e-02 & 1.52 & 3.050e-02 & 2.57 \\
468  & 2.568e-01 & 0.49 & 7.401e-02 & 1.03 & 2.500e-02 & 1.40 & 8.917e-03 & 1.79 \\
942  & 1.685e-01 & 0.60 & 4.097e-02 & 0.85 & 8.846e-03 & 1.49 & 3.174e-03 & 1.48 \\
1871 & 1.164e-01 & 0.54 & 1.984e-02 & 1.06 & 3.729e-03 & 1.26 & 9.185e-04 & 1.81 \\
3752 & 7.463e-02 & 0.64 & 9.722e-03 & 1.03 & 1.439e-03 & 1.37 & 2.894e-04 & 1.66 \\
\hline
\end{tabular}
\end{table}

%%%%%%%%%%%%%

\begin{table}[htbp]
\centering
\caption{Solving the Stokes equation in $d=2$: relative $\dot{H}^1$ errors and empirical rates.}
\label{tab:stokes_h1_rates_d2_k2_k5}
\small
\setlength{\tabcolsep}{4pt}
\begin{tabular}{c cc cc cc cc}
\hline
& \multicolumn{2}{c}{$k=2$}
& \multicolumn{2}{c}{$k=3$}
& \multicolumn{2}{c}{$k=4$}
& \multicolumn{2}{c}{$k=5$} \\
\cline{2-3}\cline{4-5}\cline{6-7}\cline{8-9}
$n$
& $\mathrm{Err}_{\dot{H}^1}^{\mathrm{rel}}$ & 0.75
& $\mathrm{Err}_{\dot{H}^1}^{\mathrm{rel}}$ & 1.25
& $\mathrm{Err}_{\dot{H}^1}^{\mathrm{rel}}$ & 1.75
& $\mathrm{Err}_{\dot{H}^1}^{\mathrm{rel}}$ & 2.25 \\
\hline
26   & 8.888e-01 & --    & 8.441e-01 & --    & 9.231e-01 & --    & 9.699e-01 & --    \\
51   & 7.161e-01 & 0.32 & 4.262e-01 & 1.01 & 5.700e-01 & 0.72 & 4.425e-01 & 1.16 \\
100  & 3.953e-01 & 0.88 & 1.741e-01 & 1.33 & 1.121e-01 & 2.42 & 8.611e-02 & 2.43 \\
202  & 2.635e-01 & 0.58 & 8.863e-02 & 0.96 & 3.093e-02 & 1.83 & 2.058e-02 & 2.04 \\
400  & 1.600e-01 & 0.73 & 3.638e-02 & 1.30 & 1.017e-02 & 1.63 & 3.360e-03 & 2.65 \\
801  & 1.005e-01 & 0.67 & 1.574e-02 & 1.21 & 3.256e-03 & 1.64 & 7.015e-04 & 2.26 \\
1604 & 5.958e-02 & 0.75 & 6.649e-03 & 1.24 & 1.033e-03 & 1.65 & 1.471e-04 & 2.25 \\
3202 & 3.495e-02 & 0.77 & 2.683e-03 & 1.31 & 2.820e-04 & 1.88 & 3.063e-05 & 2.27 \\
\hline
\end{tabular}
\end{table}

%%%%%%%%%%%%%

\begin{table}[htbp]
\centering
\caption{Solving the Stokes equation in $d=2$: relative $L^2$ errors and empirical rates.}
\label{tab:stokes_l2_rates_d2_k2_k5}
\small
\setlength{\tabcolsep}{4pt}
\begin{tabular}{c cc cc cc cc}
\hline
& \multicolumn{2}{c}{$k=2$}
& \multicolumn{2}{c}{$k=3$}
& \multicolumn{2}{c}{$k=4$}
& \multicolumn{2}{c}{$k=5$} \\
\cline{2-3}\cline{4-5}\cline{6-7}\cline{8-9}
$n$
& $\mathrm{Err}_{L^2}^{\mathrm{rel}}$ & 1.25
& $\mathrm{Err}_{L^2}^{\mathrm{rel}}$ & 1.75
& $\mathrm{Err}_{L^2}^{\mathrm{rel}}$ & 2.25
& $\mathrm{Err}_{L^2}^{\mathrm{rel}}$ & 2.75 \\
\hline
26   & 8.603e-01 & --    & 7.324e-01 & --    & 1.058e+00 & --    & 1.025e+00 & --    \\
51   & 5.817e-01 & 0.58 & 2.868e-01 & 1.39 & 4.408e-01 & 1.30 & 2.853e-01 & 1.90 \\
100  & 1.903e-01 & 1.66 & 7.597e-02 & 1.97 & 4.568e-02 & 3.37 & 3.461e-02 & 3.13 \\
202  & 8.641e-02 & 1.12 & 2.502e-02 & 1.58 & 8.509e-03 & 2.39 & 5.750e-03 & 2.55 \\
400  & 3.512e-02 & 1.32 & 7.095e-03 & 1.84 & 1.869e-03 & 2.22 & 6.393e-04 & 3.22 \\
801  & 1.505e-02 & 1.22 & 2.078e-03 & 1.77 & 4.238e-04 & 2.14 & 9.106e-05 & 2.81 \\
1604 & 5.940e-03 & 1.34 & 6.268e-04 & 1.73 & 9.770e-05 & 2.11 & 1.371e-05 & 2.73 \\
3202 & 2.301e-03 & 1.37 & 1.734e-04 & 1.86 & 1.842e-05 & 2.41 & 1.980e-06 & 2.80 \\
\hline
\end{tabular}
\end{table}

%%%%%%%%%%%

%%%%%%%%%%%

\begin{table}[htbp]
\centering
\caption{Solving the Stokes equation in $d=3$: relative $\dot{H}^1$ errors and empirical rates.}
\label{tab:stokes_h1_rates_d3_k2_k5}
\small
\setlength{\tabcolsep}{4pt}
\begin{tabular}{c cc cc cc cc}
\hline
& \multicolumn{2}{c}{$k=2$}
& \multicolumn{2}{c}{$k=3$}
& \multicolumn{2}{c}{$k=4$}
& \multicolumn{2}{c}{$k=5$} \\
\cline{2-3}\cline{4-5}\cline{6-7}\cline{8-9}
$n$
& $\mathrm{Err}_{\dot{H}^1}^{\mathrm{rel}}$ & 0.67
& $\mathrm{Err}_{\dot{H}^1}^{\mathrm{rel}}$ & 1.00
& $\mathrm{Err}_{\dot{H}^1}^{\mathrm{rel}}$ & 1.33
& $\mathrm{Err}_{\dot{H}^1}^{\mathrm{rel}}$ & 1.67 \\
\hline
30   & 9.870e-01 & --    & 9.933e-01 & --    & 9.967e-01 & --    & 9.978e-01 & --    \\
60   & 9.091e-01 & 0.12 & 8.595e-01 & 0.21 & 9.341e-01 & 0.09 & 9.740e-01 & 0.03 \\
118  & 8.086e-01 & 0.17 & 7.613e-01 & 0.18 & 7.595e-01 & 0.31 & 7.770e-01 & 0.33 \\
235  & 5.978e-01 & 0.44 & 4.323e-01 & 0.82 & 3.729e-01 & 1.03 & 3.930e-01 & 0.99 \\
468  & 4.608e-01 & 0.38 & 2.468e-01 & 0.81 & 1.703e-01 & 1.14 & 1.357e-01 & 1.54 \\
942  & 3.165e-01 & 0.54 & 1.291e-01 & 0.93 & 7.222e-02 & 1.23 & 4.374e-02 & 1.62 \\
1871 & 2.103e-01 & 0.60 & 6.843e-02 & 0.93 & 3.144e-02 & 1.21 & 1.556e-02 & 1.51 \\
3752 & 1.367e-01 & 0.62 & 3.630e-02 & 0.91 & 1.327e-02 & 1.24 & 4.981e-03 & 1.64 \\
\hline
\end{tabular}
\end{table}

%%%%%%%%%%%

\begin{table}[htbp]
\centering
\caption{Solving the Stokes equation in $d=3$: relative $L^2$ errors and empirical rates.}
\label{tab:stokes_l2_rates_d3_k2_k5}
\small
\setlength{\tabcolsep}{4pt}
\begin{tabular}{c cc cc cc cc}
\hline
& \multicolumn{2}{c}{$k=2$}
& \multicolumn{2}{c}{$k=3$}
& \multicolumn{2}{c}{$k=4$}
& \multicolumn{2}{c}{$k=5$} \\
\cline{2-3}\cline{4-5}\cline{6-7}\cline{8-9}
$n$
& $\mathrm{Err}_{L^2}^{\mathrm{rel}}$ & 1.00
& $\mathrm{Err}_{L^2}^{\mathrm{rel}}$ & 1.33
& $\mathrm{Err}_{L^2}^{\mathrm{rel}}$ & 1.67
& $\mathrm{Err}_{L^2}^{\mathrm{rel}}$ & 2.00 \\
\hline
30   & 9.759e-01 & --    & 9.817e-01 & --    & 9.947e-01 & --    & 9.994e-01 & --    \\
60   & 8.487e-01 & 0.20 & 7.754e-01 & 0.34 & 8.566e-01 & 0.22 & 9.323e-01 & 0.10 \\
118  & 7.205e-01 & 0.24 & 6.457e-01 & 0.27 & 6.544e-01 & 0.40 & 6.415e-01 & 0.55 \\
235  & 4.107e-01 & 0.82 & 2.871e-01 & 1.18 & 2.473e-01 & 1.41 & 2.643e-01 & 1.29 \\
468  & 2.578e-01 & 0.68 & 1.249e-01 & 1.21 & 8.429e-02 & 1.56 & 6.574e-02 & 2.02 \\
942  & 1.317e-01 & 0.96 & 4.767e-02 & 1.38 & 2.627e-02 & 1.67 & 1.552e-02 & 2.06 \\
1871 & 6.493e-02 & 1.03 & 1.932e-02 & 1.32 & 8.798e-03 & 1.59 & 4.359e-03 & 1.85 \\
3752 & 3.158e-02 & 1.04 & 8.000e-03 & 1.27 & 2.936e-03 & 1.58 & 1.110e-03 & 1.97 \\
\hline
\end{tabular}
\end{table}

%%%%%%%%%%%

%%%%%%%%%%%

\begin{figure}[htbp]
    \centering
    \includegraphics[width=\linewidth]{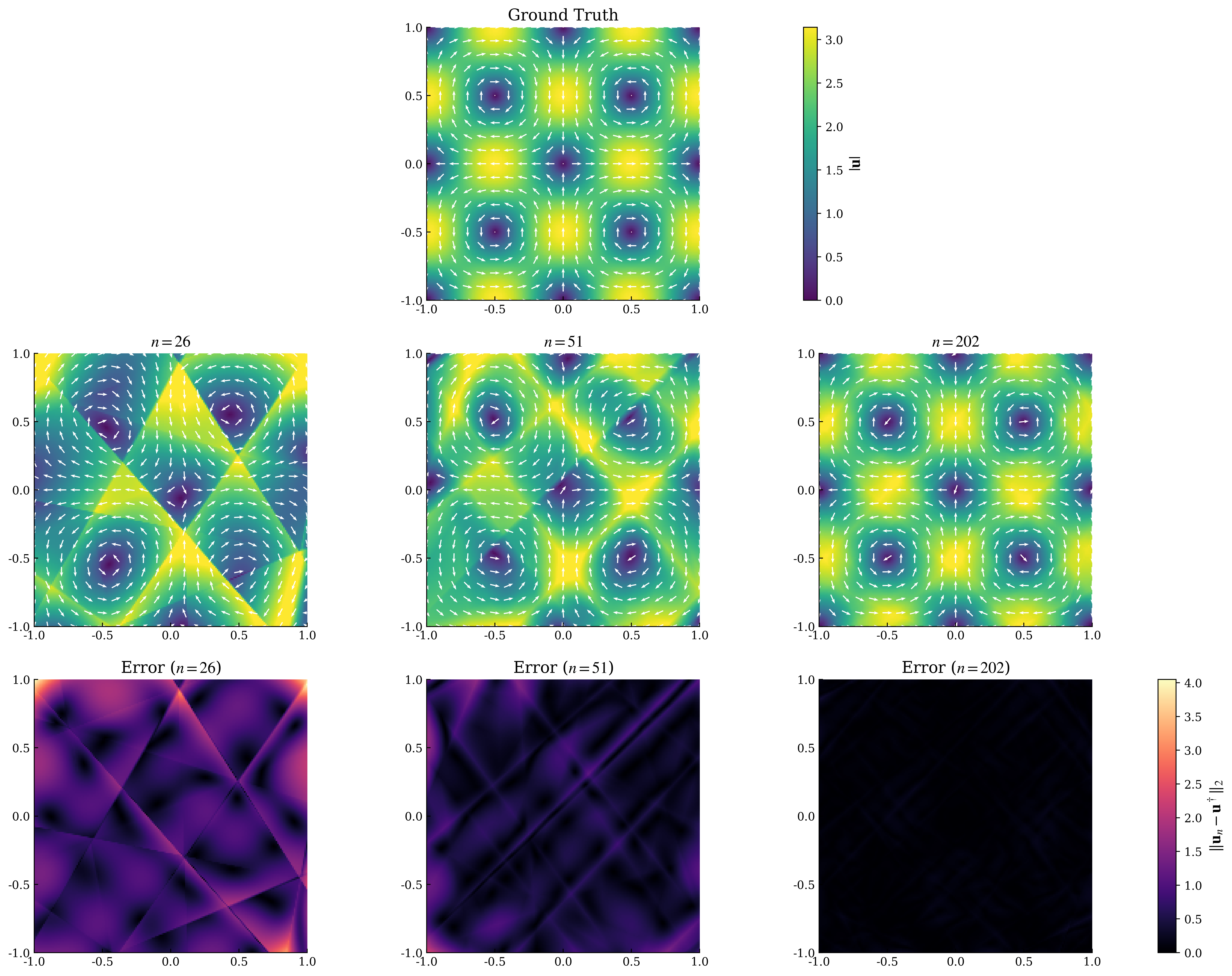}
    \caption{divergence-free $L^2$ approximation when $d=2, k=2$
    \emph{Top:} ground-truth divergence-free field $\mathbf{u}^\dagger$.
    \emph{Middle:} learned approximations $\mathbf{u}_n$ from three checkpoints with increasing neuron count $n$ (left to right).
    \emph{Bottom:} pointwise error heatmaps $e(x)=\|\mathbf{u}_n(x)-\mathbf{u}^\dagger(x)\|_2$ for the same three models.
    In each velocity panel, the background color encodes the speed $|\mathbf{u}(x)|$, while arrows indicate the flow direction.}
    \label{fig:divfree2d-flow}
\end{figure}

\begin{figure}[htbp]
    \centering
    \includegraphics[width=\linewidth]{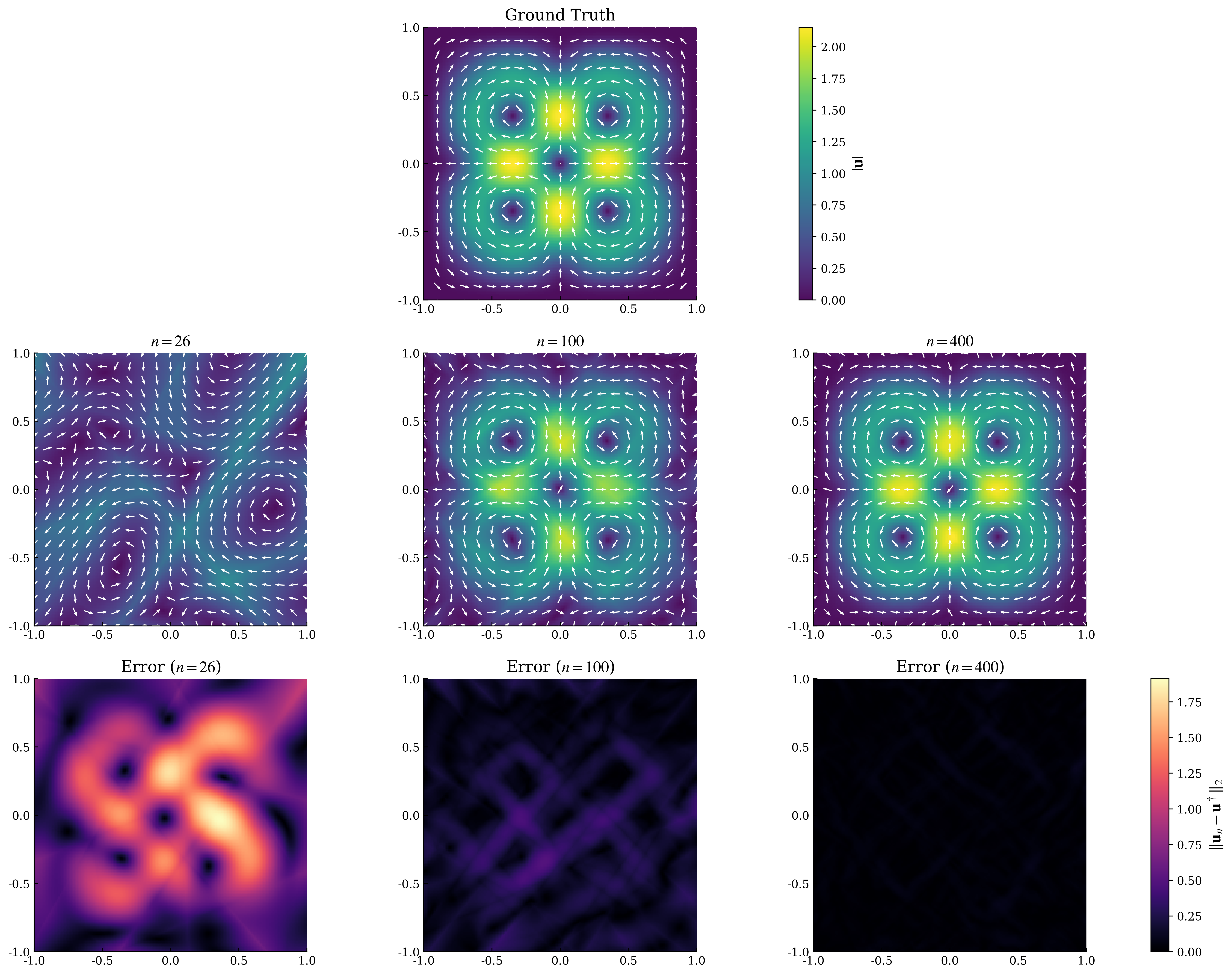}
    \caption{stokes equaition when $d=2, k=2$.
    \emph{Top:} reference Stokes velocity field $\mathbf{u}^\dagger$.
    \emph{Middle:} learned divergence-free approximations $\mathbf{u}_n$ with increasing neuron count $n$ (left to right).
    \emph{Bottom:} pointwise error heatmaps $e(x)=\|\mathbf{u}_n(x)-\mathbf{u}^\dagger(x)\|_2$.
    In each velocity panel, the background color encodes the speed $|\mathbf{u}(x)|$, and arrows indicate the velocity direction.}
    \label{fig:stokes2d-flow}
\end{figure}

%%%%%%%%%%%%%%%%%%%%%%%%%%%%%%%

\begin{table}[htbp]
\centering
\caption{Solving classical lid-driven cavity flow by divergence-free FNS: errors and empirical rates for $k=3$.}
\label{tab:classic_lid_k3}
\small
\setlength{\tabcolsep}{4pt}

\begin{minipage}[t]{0.48\textwidth}
\centering
\textbf{(a) $L^2$ errors}

\vspace{0.3em}
\begin{tabular}{c cc cc}
\hline
& \multicolumn{2}{c}{global}
& \multicolumn{2}{c}{inner} \\
\cline{2-3}\cline{4-5}
$n$
& $\mathrm{Err}_{L^2}$ & --
& $\mathrm{Err}_{L^2}^{\mathrm{inner}}$ & -- \\
\hline
51   & 4.669e-01 & --   & 4.495e-01 & --   \\
100  & 2.090e-01 & 1.19 & 1.859e-01 & 1.31 \\
202  & 1.036e-01 & 1.00 & 8.876e-02 & 1.05 \\
400  & 7.526e-02 & 0.47 & 6.377e-02 & 0.48 \\
801  & 5.430e-02 & 0.47 & 4.070e-02 & 0.65 \\
1604 & 4.366e-02 & 0.31 & 3.264e-02 & 0.32 \\
3202 & 3.709e-02 & 0.24 & 2.837e-02 & 0.20 \\
\hline
\end{tabular}
\end{minipage}
\hfill
\begin{minipage}[t]{0.48\textwidth}
\centering
\textbf{(b) $\dot H^1$ errors}

\vspace{0.3em}
\begin{tabular}{c cc cc}
\hline
& \multicolumn{2}{c}{global}
& \multicolumn{2}{c}{inner} \\
\cline{2-3}\cline{4-5}
$n$
& $\mathrm{Err}_{\dot H^1}$ & --
& $\mathrm{Err}_{\dot H^1}^{\mathrm{inner}}$ & -- \\
\hline
51   & 5.361e+00 & --   & 2.058e+00 & --   \\
100  & 5.080e+00 & 0.08 & 1.810e+00 & 0.19 \\
202  & 4.604e+00 & 0.14 & 1.395e+00 & 0.37 \\
400  & 4.401e+00 & 0.07 & 1.417e+00 & -0.02 \\
801  & 4.293e+00 & 0.04 & 9.059e-01 & 0.64 \\
1604 & 4.081e+00 & 0.07 & 7.808e-01 & 0.21 \\
3202 & 3.777e+00 & 0.11 & 5.513e-01 & 0.50 \\
\hline
\end{tabular}
\end{minipage}

\end{table}

%========================================================
% lid-driven cavity flow: Error heat maps
%========================================================

\begin{figure}[htbp]
  \centering
    \includegraphics[width=0.45\textwidth]{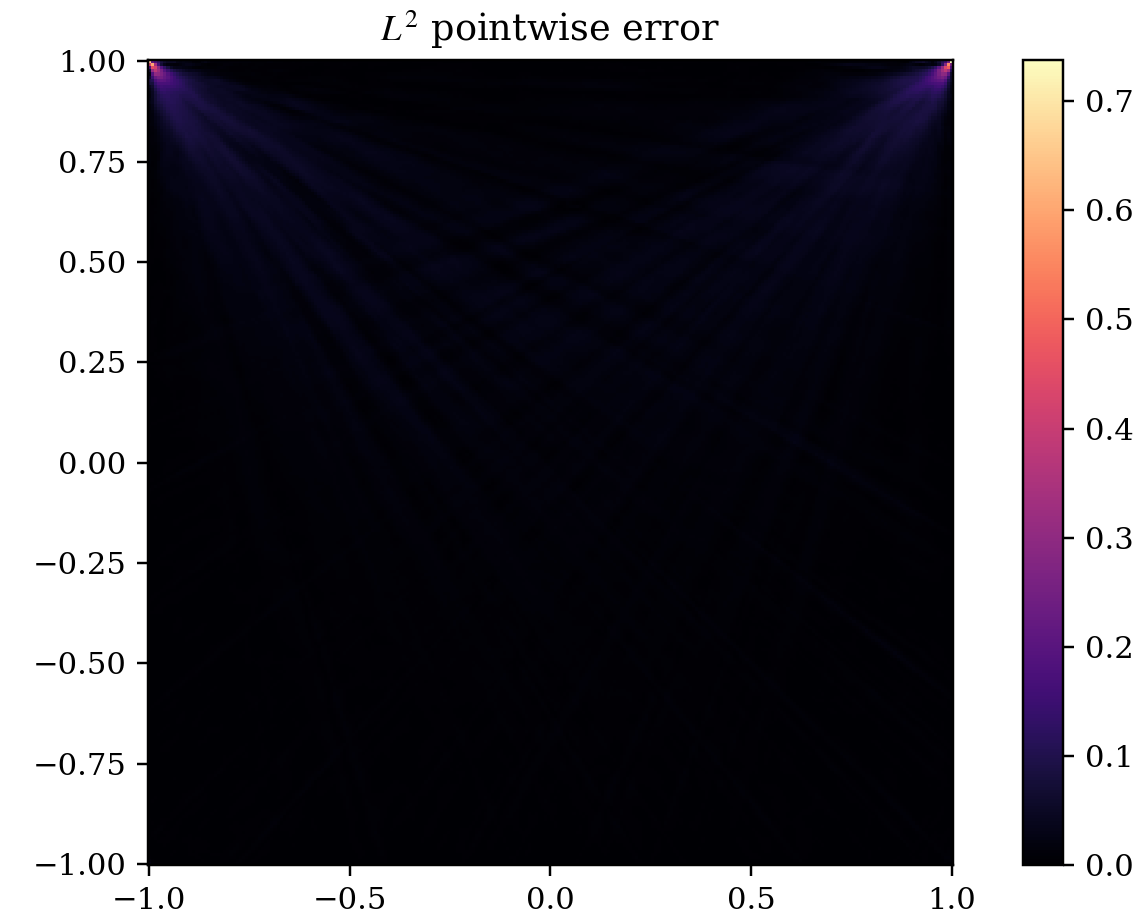}
    \includegraphics[width=0.45\textwidth]{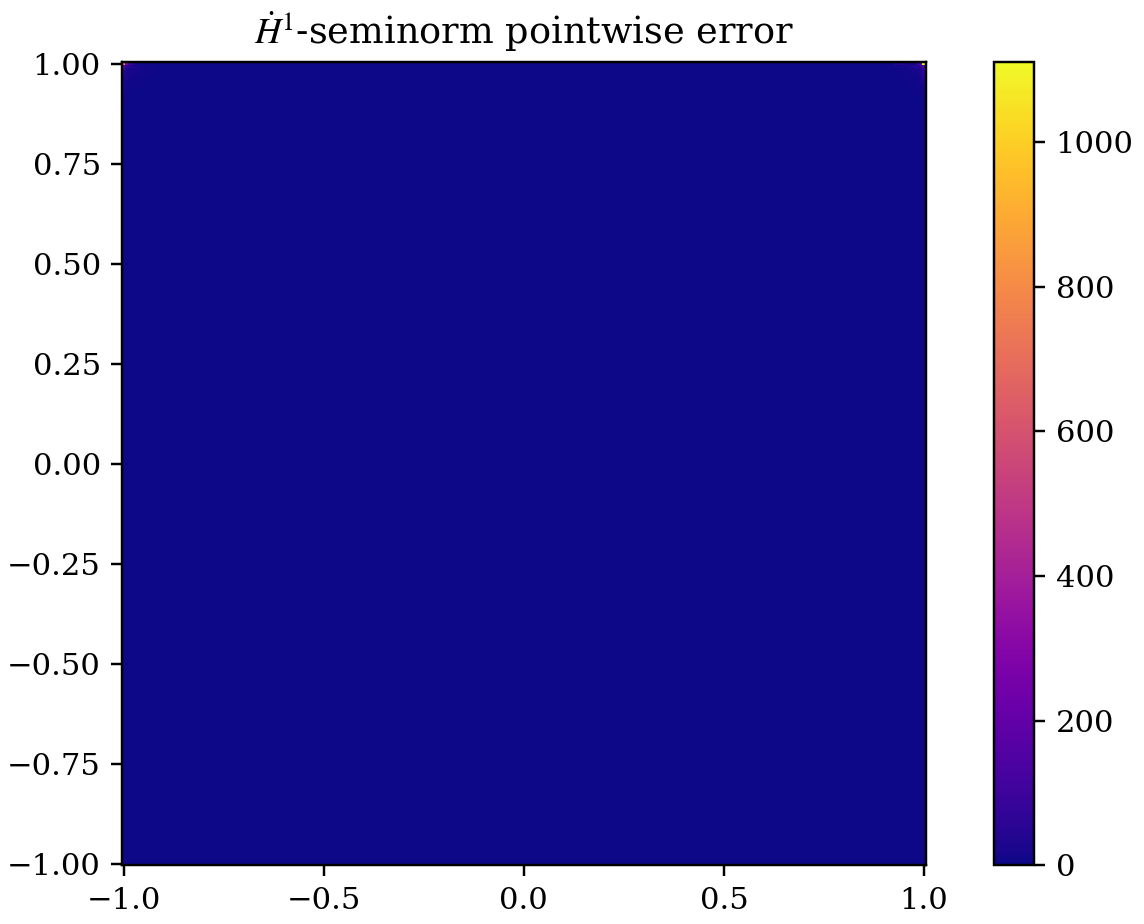}
    \includegraphics[width=0.45\textwidth]{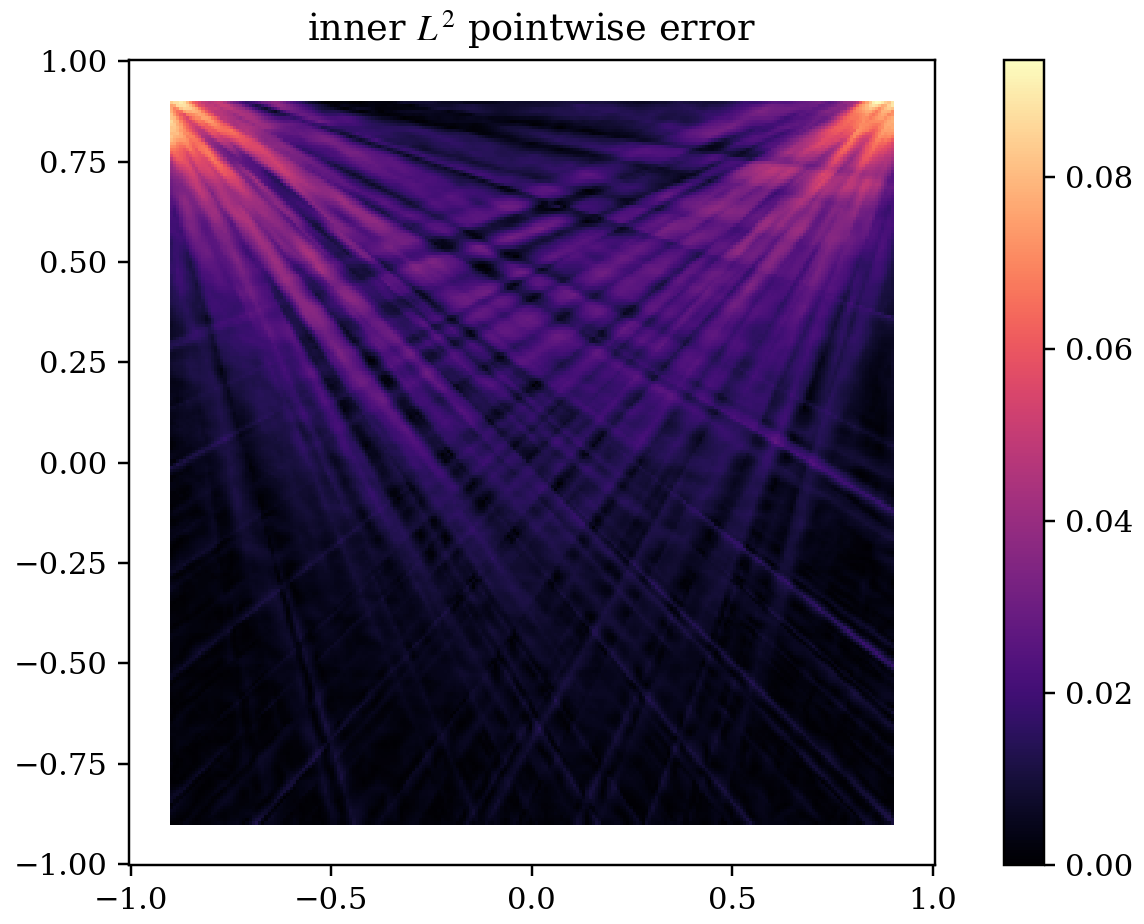}
    \includegraphics[width=0.45\textwidth]{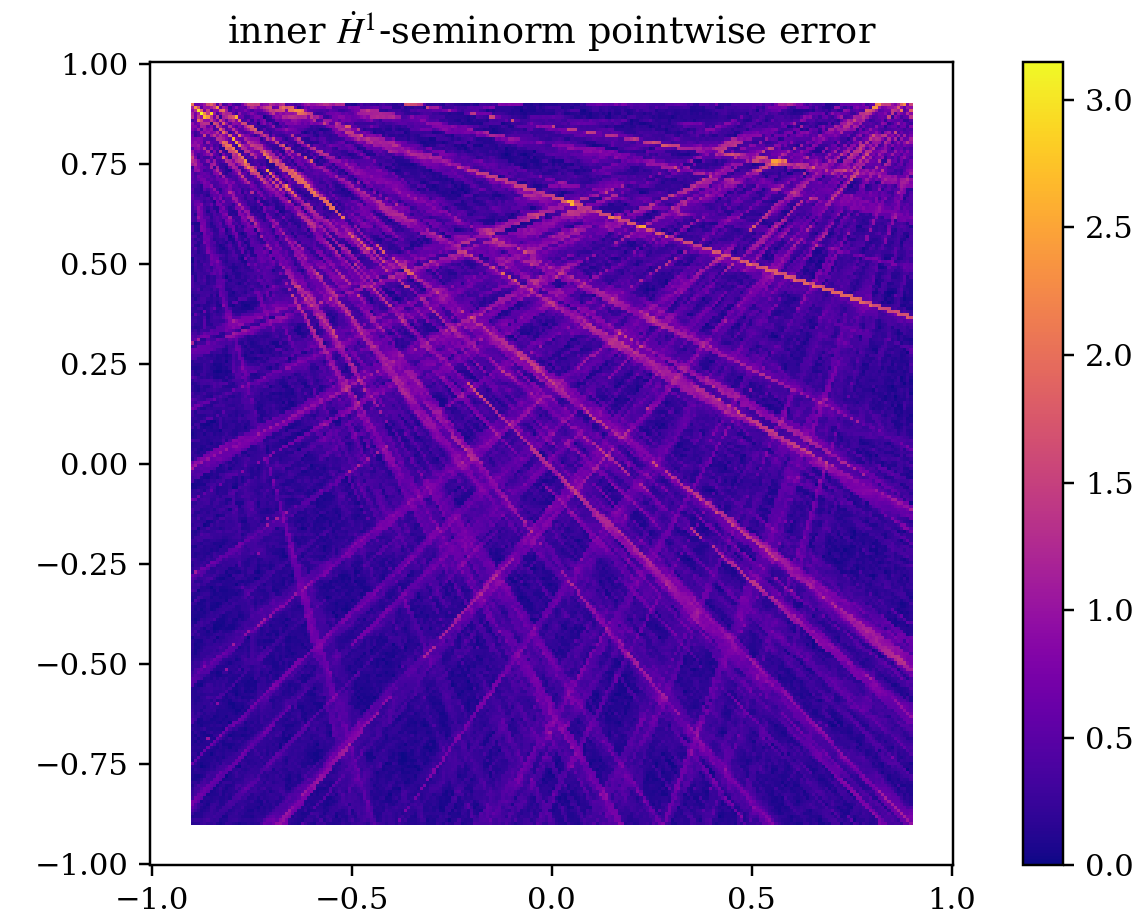}
  \caption{Solving classical lid-driven cavity flow by divergence-free FNS, $k=2, n=3202$. Restricted evaluation on $[-0.9, 0.9]^2$ shows that the large errors occur near the top two corners, especially for the $\dot H^1$ error.}
  \label{fig:classic-lid-error-heatmap}
\end{figure}

\begin{figure}[htbp]
  \centering
  \includegraphics[width=0.45\textwidth]{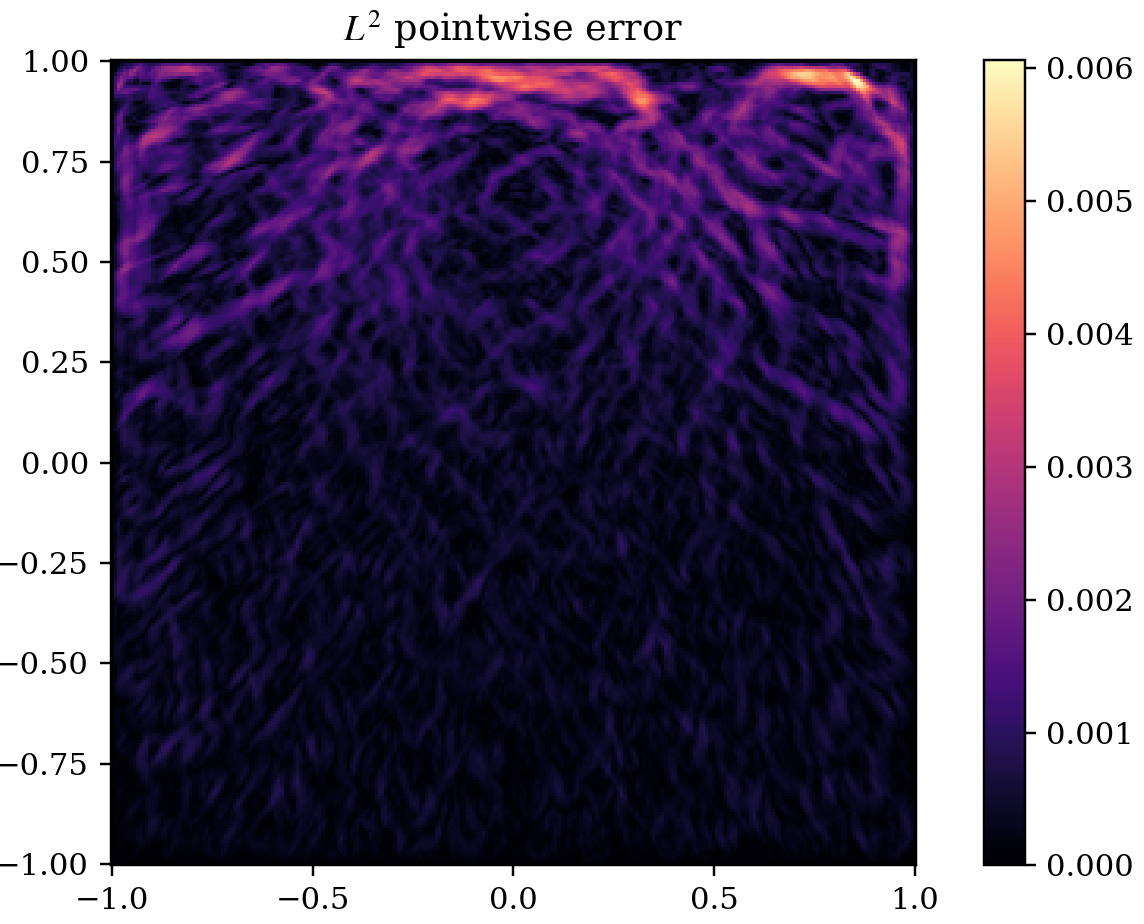}
  \includegraphics[width=0.45\textwidth]{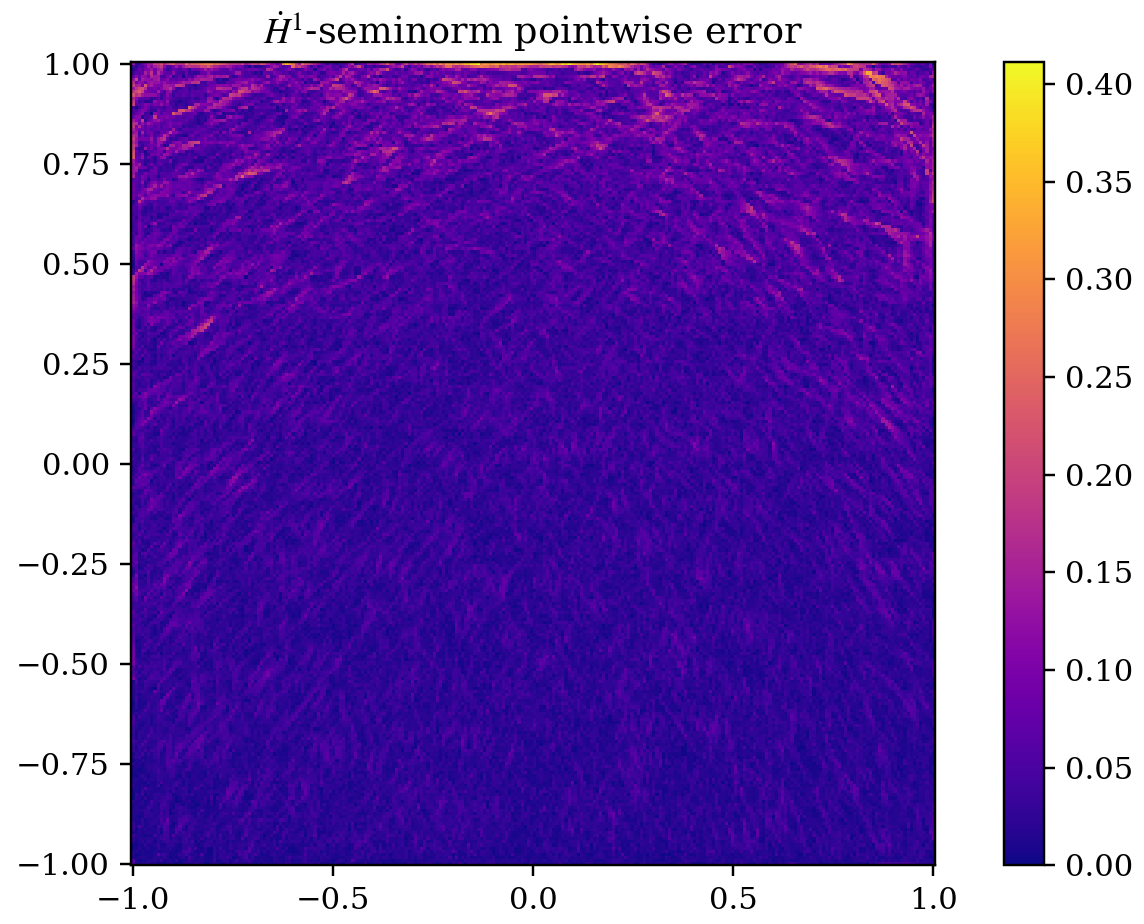}
  \caption{Solving regularized lid-driven cavity flow by divergence-free FNS,  $k=2, n=3202$. Different from Figure~\ref{fig:classic-lid-error-heatmap}, no large error spike here due to the regularized lid.}
  \label{fig:regularized-lid-error-heatmap}
\end{figure}

%%%%%%%%%%%%%%%%%%%%%%%%%%%%%%%%%%

%%%%%%%%%%%%% table: regularized lid %%%%%%%%%%%%%
\begin{table}[htbp]
\centering
\caption{Regularized lid-driven cavity flow: errors and empirical rates.}
\label{tab:reg_lid_k2_k3}
\small
\setlength{\tabcolsep}{4pt}

\begin{minipage}[t]{0.48\textwidth}
\centering
\textbf{(a) $L^2$ errors}

\vspace{0.3em}
\begin{tabular}{c cc cc}
\hline
& \multicolumn{2}{c}{$k=2$}
& \multicolumn{2}{c}{$k=3$} \\
\cline{2-3}\cline{4-5}
$n$
& $\mathrm{Err}_{L^2}$ & 1.25
& $\mathrm{Err}_{L^2}$ & 1.75 \\
\hline
51   & 5.241e-01 & --    & 6.084e-01 & --    \\
100  & 3.863e-01 & 0.45 & 1.311e-01 & 2.28 \\
202  & 1.769e-01 & 1.11 & 2.864e-02 & 2.16 \\
400  & 5.080e-02 & 1.83 & 5.531e-03 & 2.41 \\
801  & 1.731e-02 & 1.55 & 1.366e-03 & 2.01 \\
1604 & 5.421e-03 & 1.67 & 3.180e-04 & 2.10 \\
3202 & 1.708e-03 & 1.67 & 1.034e-04 & 1.62 \\
\hline
\end{tabular}
\end{minipage}
\hfill
\begin{minipage}[t]{0.48\textwidth}
\centering
\textbf{(b) $\dot H^1$ errors}

\vspace{0.3em}
\begin{tabular}{c cc cc}
\hline
& \multicolumn{2}{c}{$k=2$}
& \multicolumn{2}{c}{$k=3$} \\
\cline{2-3}\cline{4-5}
$n$
& $\mathrm{Err}_{\dot H^1}$ & 0.75
& $\mathrm{Err}_{\dot H^1}$ & 1.25 \\
\hline
51   & 3.120e+00 & --    & 2.483e+00 & --    \\
100  & 2.162e+00 & 0.54 & 8.205e-01 & 1.64 \\
202  & 1.196e+00 & 0.84 & 3.383e-01 & 1.26 \\
400  & 6.310e-01 & 0.94 & 1.194e-01 & 1.52 \\
801  & 3.375e-01 & 0.90 & 4.886e-02 & 1.29 \\
1604 & 1.823e-01 & 0.89 & 1.744e-02 & 1.48 \\
3202 & 1.022e-01 & 0.84 & 7.140e-03 & 1.29 \\
\hline
\end{tabular}
\end{minipage}

\end{table}

%========================================================
% classic lid-driven cavity flow figures: k=3
%========================================================

\begin{figure}[htbp]
  \centering
  \includegraphics[width=0.45\textwidth]{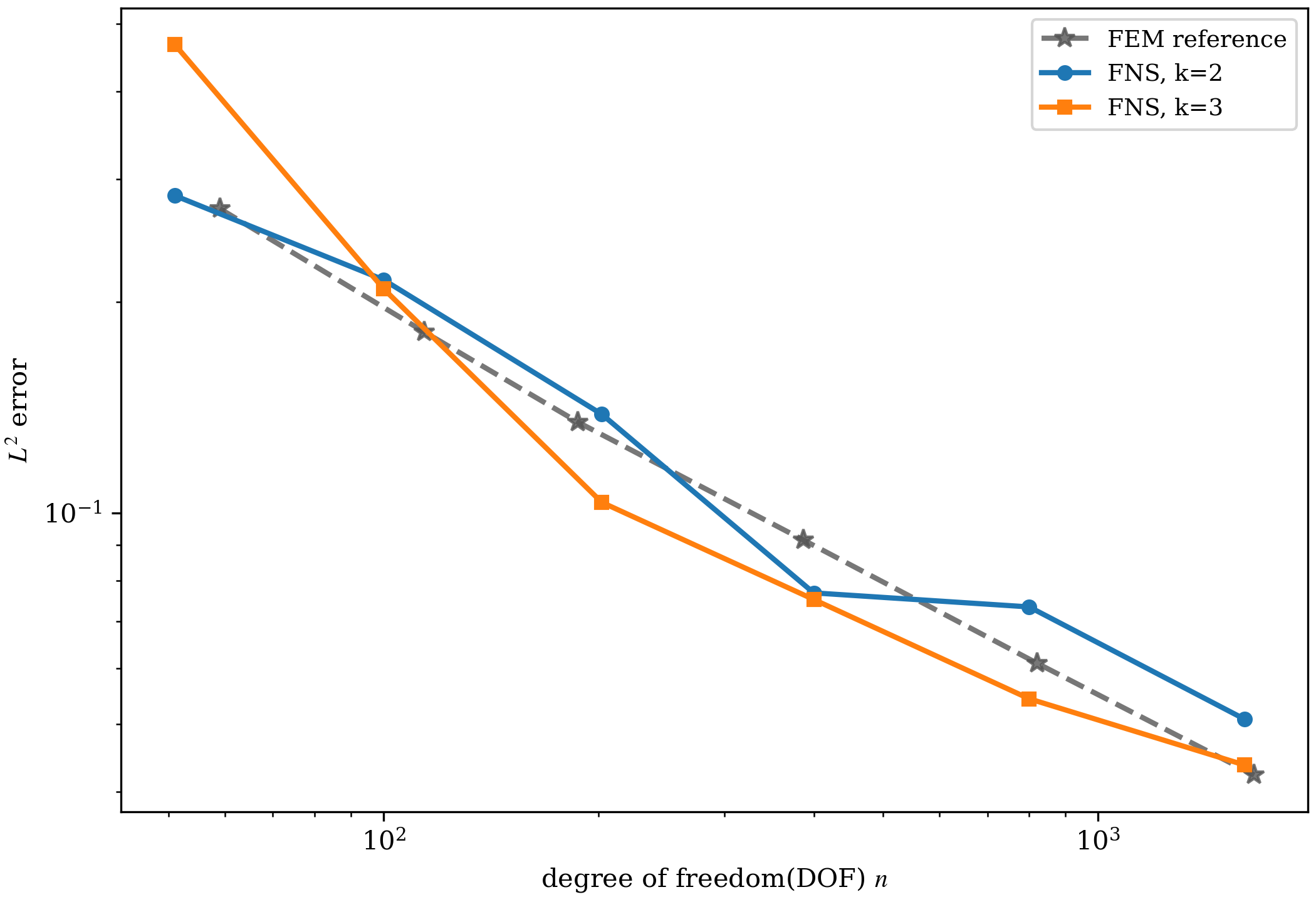}
  \includegraphics[width=0.45\textwidth]{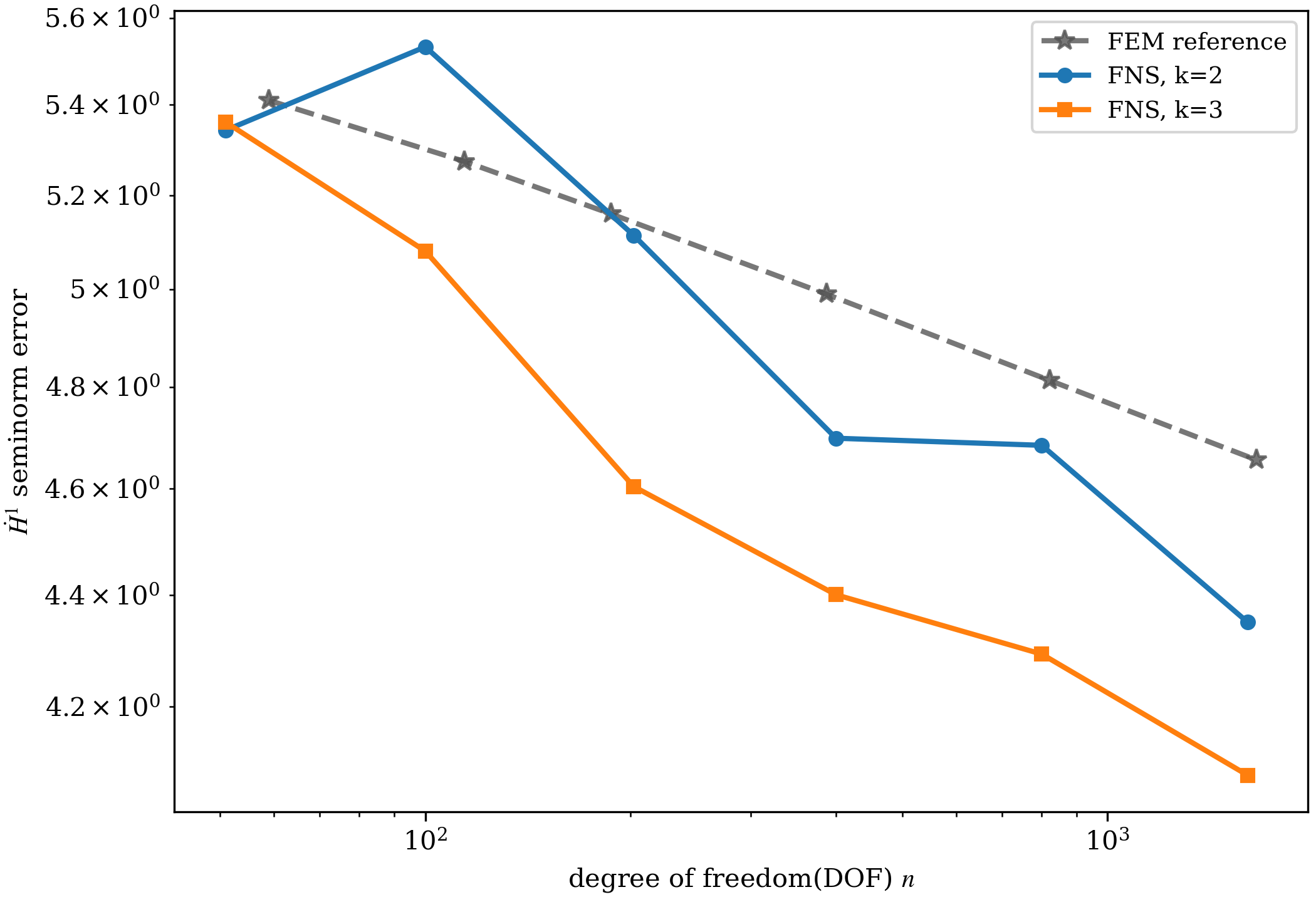}
  \caption{Solving classical lid-driven cavity flow by divergence-free FNS compared with finite element method: $L^2$ error and $\dot{H}^1$ seminorm error decay with degree of freedom.
  Left: $L^2$ error. Right: $\dot{H}^1$ seminorm error.}
  \label{fig:classic-lid-fem&fns-compare}
\end{figure}
%%%%%%%%%

%%%%%%%%%%%%%%%%%%%% boundary ablations %%%%%%%%%%%%%%%%%%%%

\begin{table}[htbp]
\centering
\caption{Sensitivity analysis of the Stokes problem in $d=2$ with $k=4$. Panels (a)--(c) report the results for different boundary penalty parameters $\lambda_{\partial\Omega}$.}
\label{tab:stokes2d_lambda_sensitivity}
\small
\setlength{\tabcolsep}{4pt}

\begin{minipage}[t]{\textwidth}
\centering
\textbf{(a) $\lambda_{\partial\Omega}=0.01$}

\vspace{0.3em}
\begin{tabular}{c c cc cc}
\hline
$n$
& $\mathrm{Err}_{\partial\Omega}$
& $\mathrm{Err}_{\dot{H}^1}^{\mathrm{rel}}$ & 1.75
& $\mathrm{Err}_{L^2}^{\mathrm{rel}}$ & 2.25 \\
\hline
26   & 8.4827e+00 & 8.9048e-01 & --      & 3.1491e+00 & --      \\
51   & 2.1631e+00 & 5.6881e-01 & 0.6653 & 6.7923e-01 & 2.2767 \\
100  & 1.8038e-01 & 1.1204e-01 & 2.4128 & 4.8687e-02 & 3.9141 \\
202  & 3.5616e-02 & 3.0930e-02 & 1.8307 & 8.5358e-03 & 2.4764 \\
400  & 8.6296e-03 & 1.0169e-02 & 1.6282 & 1.8686e-03 & 2.2235 \\
801  & 1.9233e-03 & 3.2562e-03 & 1.6400 & 4.2381e-04 & 2.1366 \\
1604 & 4.5605e-04 & 1.0327e-03 & 1.6538 & 9.7701e-05 & 2.1132 \\
3202 & 7.3642e-05 & 2.8202e-04 & 1.8777 & 1.8423e-05 & 2.4134 \\
\hline
\end{tabular}
\end{minipage}

\vspace{0.8em}

\begin{minipage}[t]{\textwidth}
\centering
\textbf{(b) $\lambda_{\partial\Omega}=1$}

\vspace{0.3em}
\begin{tabular}{c c cc cc}
\hline
$n$
& $\mathrm{Err}_{\partial\Omega}$
& $\mathrm{Err}_{\dot{H}^1}^{\mathrm{rel}}$ & 1.75
& $\mathrm{Err}_{L^2}^{\mathrm{rel}}$ & 2.25 \\
\hline
26   & 1.6330e+00 & 9.2313e-01 & --      & 1.0584e+00 & --      \\
51   & 1.5296e+00 & 5.6999e-01 & 0.7157 & 4.4079e-01 & 1.3001 \\
100  & 1.6745e-01 & 1.1208e-01 & 2.4154 & 4.5680e-02 & 3.3666 \\
202  & 3.4319e-02 & 3.0935e-02 & 1.8309 & 8.5089e-03 & 2.3902 \\
400  & 8.4234e-03 & 1.0170e-02 & 1.6283 & 1.8689e-03 & 2.2186 \\
801  & 1.8885e-03 & 3.2563e-03 & 1.6400 & 4.2381e-04 & 2.1368 \\
1604 & 4.4975e-04 & 1.0327e-03 & 1.6538 & 9.7695e-05 & 2.1133 \\
3202 & 7.2954e-05 & 2.8202e-04 & 1.8777 & 1.8423e-05 & 2.4133 \\
\hline
\end{tabular}
\end{minipage}

\vspace{0.8em}

\begin{minipage}[t]{\textwidth}
\centering
\textbf{(c) $\lambda_{\partial\Omega}=100$}

\vspace{0.3em}
\begin{tabular}{c c cc cc}
\hline
$n$
& $\mathrm{Err}_{\partial\Omega}$
& $\mathrm{Err}_{\dot{H}^1}^{\mathrm{rel}}$ & 1.75
& $\mathrm{Err}_{L^2}^{\mathrm{rel}}$ & 2.25 \\
\hline
26   & 1.4388e-01 & 9.8421e-01 & --      & 9.7238e-01 & --      \\
51   & 4.2847e-01 & 7.6651e-01 & 0.3711 & 7.1721e-01 & 0.4518 \\
100  & 9.6074e-02 & 1.2375e-01 & 2.7083 & 5.7570e-02 & 3.7460 \\
202  & 1.7822e-02 & 3.3002e-02 & 1.8798 & 9.9664e-03 & 2.4943 \\
400  & 4.6739e-03 & 1.0589e-02 & 1.6639 & 2.0391e-03 & 2.3225 \\
801  & 1.0859e-03 & 3.3271e-03 & 1.6672 & 4.4238e-04 & 2.2006 \\
1604 & 2.5461e-04 & 1.0465e-03 & 1.6657 & 9.9876e-05 & 2.1432 \\
3202 & 4.8178e-05 & 2.8323e-04 & 1.8906 & 1.8592e-05 & 2.4321 \\
\hline
\end{tabular}
\end{minipage}

\end{table}

%%%%%%%%%

\begin{table}[htbp]
\centering
\caption{Sensitivity analysis of the Stokes problem in $d=3$ with $k=4$. Panels (a)--(c) report the results for different boundary penalty parameters $\lambda_{\partial\Omega}$.}
\label{tab:stokes3d_lambda_sensitivity}
\small
\setlength{\tabcolsep}{4pt}

\begin{minipage}[t]{\textwidth}
\centering
\textbf{(a) $\lambda_{\partial\Omega}=0.06$}

\vspace{0.3em}
\begin{tabular}{c c cc cc}
\hline
$n$ & $\mathrm{Err}_{\partial\Omega}$ & $\mathrm{Err}_{\dot{H}^1}^{\mathrm{rel}}$ & 1.33 & $\mathrm{Err}_{L^2}^{\mathrm{rel}}$ & 1.67 \\
\hline
30   & 9.8541e-01 & 9.9546e-01 & --      & 1.0053e+00 & --      \\
60   & 2.3172e+00 & 9.2277e-01 & 0.1094 & 8.1890e-01 & 0.2958 \\
118  & 2.3789e+00 & 7.4795e-01 & 0.3106 & 6.3375e-01 & 0.3789 \\
235  & 1.3846e+00 & 3.6676e-01 & 1.0344 & 2.3825e-01 & 1.4202 \\
468  & 5.2821e-01 & 1.6852e-01 & 1.1288 & 8.1871e-02 & 1.5506 \\
942  & 1.8737e-01 & 7.1684e-02 & 1.2220 & 2.5728e-02 & 1.6548 \\
1871 & 6.4159e-02 & 3.1326e-02 & 1.2064 & 8.6704e-03 & 1.5850 \\
3752 & 2.1954e-02 & 1.3236e-02 & 1.2381 & 2.9120e-03 & 1.5680 \\
\hline
\end{tabular}
\end{minipage}

\vspace{0.8em}

\begin{minipage}[t]{\textwidth}
\centering
\textbf{(b) $\lambda_{\partial\Omega}=6$}

\vspace{0.3em}
\begin{tabular}{c c cc cc}
\hline
$n$ & $\mathrm{Err}_{\partial\Omega}$ & $\mathrm{Err}_{\dot{H}^1}^{\mathrm{rel}}$ & 1.33 & $\mathrm{Err}_{L^2}^{\mathrm{rel}}$ & 1.67 \\
\hline
30   & 3.4182e-01 & 9.9674e-01 & --      & 9.9472e-01 & --      \\
60   & 1.3690e+00 & 9.3408e-01 & 0.0937 & 8.5665e-01 & 0.2156 \\
118  & 1.7355e+00 & 7.5952e-01 & 0.3059 & 6.5435e-01 & 0.3983 \\
235  & 1.1178e+00 & 3.7294e-01 & 1.0325 & 2.4729e-01 & 1.4125 \\
468  & 4.4215e-01 & 1.7027e-01 & 1.1381 & 8.4287e-02 & 1.5624 \\
942  & 1.5972e-01 & 7.2215e-02 & 1.2262 & 2.6271e-02 & 1.6665 \\
1871 & 5.6090e-02 & 3.1444e-02 & 1.2116 & 8.7984e-03 & 1.5941 \\
3752 & 1.9626e-02 & 1.3274e-02 & 1.2393 & 2.9362e-03 & 1.5774 \\
\hline
\end{tabular}
\end{minipage}

\vspace{0.8em}

\begin{minipage}[t]{\textwidth}
\centering
\textbf{(c) $\lambda_{\partial\Omega}=600$}

\vspace{0.3em}
\begin{tabular}{c c cc cc}
\hline
$n$ & $\mathrm{Err}_{\partial\Omega}$ & $\mathrm{Err}_{\dot{H}^1}^{\mathrm{rel}}$ & 1.33 & $\mathrm{Err}_{L^2}^{\mathrm{rel}}$ & 1.67 \\
\hline
30   & 7.4446e-03 & 9.9990e-01 & --      & 9.9982e-01 & --      \\
60   & 3.9988e-02 & 9.9715e-01 & 0.0040 & 9.9469e-01 & 0.0074 \\
118  & 1.2912e-01 & 9.6835e-01 & 0.0433 & 9.6005e-01 & 0.0524 \\
235  & 3.2078e-01 & 6.5845e-01 & 0.5599 & 6.1007e-01 & 0.6582 \\
468  & 1.9790e-01 & 2.5869e-01 & 1.3562 & 1.7749e-01 & 1.7922 \\
942  & 7.0888e-02 & 1.0173e-01 & 1.3342 & 4.7481e-02 & 1.8850 \\
1871 & 2.3864e-02 & 4.0748e-02 & 1.3332 & 1.3921e-02 & 1.7880 \\
3752 & 8.2437e-03 & 1.6292e-02 & 1.3175 & 4.1505e-03 & 1.7390 \\
\hline
\end{tabular}
\end{minipage}

\end{table}

%%%%%%%%%%%%%%%%%%%%%%%%%%%%%

\section*{Acknowledgement}
The authors would like to thank the THU-ATOM Lab at the Institute for AI Industry Research (AIR), Tsinghua University, for providing the computational resources that supported this research.

\newpage
\bibliographystyle{plainnat}
\bibliography{references}

\end{document}